\documentclass[review,hidelinks,onefignum,onetabnum]{siamart251104}

\usepackage{lipsum}
\usepackage{amsfonts}
\usepackage{graphicx}
\usepackage{epstopdf}
\usepackage{algorithmic}
\ifpdf
  \DeclareGraphicsExtensions{.eps,.pdf,.png,.jpg}
\else
  \DeclareGraphicsExtensions{.eps}
\fi

\newsiamremark{remark}{Remark}
\newsiamremark{hypothesis}{Hypothesis}
\crefname{hypothesis}{Hypothesis}{Hypotheses}
\newsiamthm{claim}{Claim}
\newsiamremark{fact}{Fact}
\crefname{fact}{Fact}{Facts}

\headers{Invariant-domain-preserving Limiter for Ideal MHD}{C. Liu, C.-W. Shu, and X. Zhang}

\title{Efficient Admissible Set Projection in Optimization-based Invariant-Domain-Preserving Limiters for Ideal MHD\thanks{Submitted to the editors DATE.
\funding{C. Liu is partially supported by the NSF grant DMS-2513106. C.-W. Shu is partially supported by NSF grant DMS-2309249.
X. Zhang is partially supported by NSF grant DMS-2208518.}}}

\author{Chen Liu\thanks{Department of Mathematical Sciences, University of Arkansas, Fayetteville, Arkansas 72701, USA. (\email{chenl@uark.edu})}
\and Chi-Wang Shu\thanks{Division of Applied Mathematics, Brown University, Providence, Rhode Island 02912, USA. (\email{chi-wang\_shu@brown.edu})}
\and Xiangxiong Zhang\thanks{Department of Mathematics, Purdue University, West Lafayette, Indiana 47907, USA. (\email{zhan1966@purdue.edu})}
}

\usepackage{amsopn}

\ifpdf
\hypersetup{
  pdftitle={Efficient Admissible Set Projection in Optimization-based Invariant-domain-preserving Limiters for Ideal MHD},
  pdfauthor={C. Liu, C.-W. Shu, and X. Zhang}
}
\fi

\usepackage{xcolor}

\usepackage[cmintegrals]{newpxmath}  

\usepackage{bm}                      

\usepackage{dsfont}                                              
  \newcommand{\IR}{\ensuremath\mathds{R}}                        
  \newcommand{\IP}{\ensuremath\mathds{P}}                        

\newcommand*{\setE}{\ensuremath{\mathcal{T}}}                    
\renewcommand*{\vec}[1]{{\boldsymbol{#1}}}                       
\DeclareMathAlphabet{\mathbfsf}{\encodingdefault}{\sfdefault}{bx}{n}
\newcommand*{\vecc}[1]{\mathbfsf{#1}}                            
\newcommand*{\transpose}[1]{{#1}^\mathrm{T}}                     
\newcommand*{\normal}{\vec{n}}                                   
\newcommand*{\grad}{\vec{\nabla}}                                
\renewcommand*{\div}{\vec{\nabla}\cdot}                          

\newcommand*{\abs}[1]{\ensuremath{|#1|}}                         
\newcommand*{\norm}[2]{\|#1\|_{#2}}                              
\newcommand*{\on}[2]{\left.#1\right\vert_{#2}}                   
\newcommand*{\prox}{\mathtt{prox}}                               
\newcommand*{\pound}{\raisebox{0.1ex}{\scalebox{0.9}{\#}}}       

\usepackage{longtable}
\usepackage{multirow}
\usepackage{tabularx}                                            
  \newcolumntype{R}{>{\raggedleft\arraybackslash}X}
  \newcolumntype{L}{>{\raggedright\arraybackslash}X}
  \newcolumntype{C}{>{\centering\arraybackslash}X}
\usepackage{booktabs}                                            
\usepackage{rotating}                                            

\usepackage{enumitem}                                            

\begin{document}

\maketitle

\begin{abstract}
Preserving the admissible set of the ideal magnetohydrodynamics (MHD) equations is important not only for producing physically meaningful numerical solutions, but more importantly for achieving robust computations. In this paper, we develop an optimization-based limiter to enforce admissibility while preserving global conservation and accuracy. For an easy and efficient projection, we decompose the admissible set into slices parameterized by the magnetic energy, so that the MHD projection reduces to a one-dimensional minimization, which can be solved efficiently by the Brent method. The splitting method can be used to efficiently solve the global minimization problem of the optimization-based limiter, which can be used to enforce cell average admissibility in discontinuous Galerkin (DG) schemes, and pointwise admissibility can be further enforced by the Zhang--Shu positivity-preserving limiter. We apply the limiter to high-order DG schemes and present numerical results for a few representative MHD problems.

\end{abstract}

\begin{keywords}
MHD equations; invariant-domain-preserving; optimization-based limiter; Davis--Yin splitting; discontinuous Galerkin
\end{keywords}

\begin{MSCcodes}
65K05, 65K10, 65M60, 90C25
\end{MSCcodes}


\section{Introduction}
We consider the ideal compressible magnetohydrodynamics (MHD) system for a perfectly conducting fluid without viscosity or thermal conductivity on a bounded spatial domain $\Omega\subset\IR^n$ over the time interval $[0, T]$. The unknown conservative variables are density $\rho$, momentum $\vec{m}$, total energy $E$, and magnetic field $\vec{B}$, which satisfy
\begin{align}\label{eq:MHD_system}
\partial_t\begin{bmatrix}\rho \\ \vec{m} \\ E \\ \vec{B}\end{bmatrix} 
+ \div\begin{bmatrix}
\vec{m} \\
\vec{m}\otimes\vec{u} - \vec{B}\otimes\vec{B} + p_\mathrm{tot}\vecc{I} \\
(E + p_\mathrm{tot})\vec{u} - (\vec{u}\cdot\vec{B})\vec{B}\\
\vec{u}\otimes\vec{B} - \vec{B}\otimes\vec{u} 
\end{bmatrix} =
\begin{bmatrix}0 \\ \vec{0} \\ 0 \\ \vec{0}\end{bmatrix}.
\end{align}
We order the unknowns as $(\rho, \vec{m}, E, \vec{B})$, placing $\vec{B}$ last. This ordering is convenient for the set decomposition in Section~\ref{sec:projection}, where we slice the admissible set by fixing $\norm{\vec{B}}{2}^2$.
The solenoidal condition $\div{\vec{B}} = 0$ is imposed on the magnetic field.
The velocity is computed by $\vec{u} = \frac{\vec{m}}{\rho}$. The total energy is $E = \rho e + \frac{\norm{\vec{m}}{2}^2}{2\rho} + \frac{\norm{\vec{B}}{2}^2}{2}$, which consists of thermal, kinetic, and magnetic energies. Here, $e$ denotes the specific internal energy. 
The total pressure $p_\mathrm{tot} = p+\frac{1}{2}\norm{\vec{B}}{2}^2$, which is the sum of the gas pressure and magnetic pressure. The gas pressure is determined by ideal gas equation of state $p = (\gamma-1)\rho e$, where the adiabatic index $\gamma > 1$.
\par
Physically meaningful solutions to the MHD system \eqref{eq:MHD_system} should have positive density and positive internal energy. We define the admissible state set:
\begin{align*}
G = \Big\{(\rho, \vec{m}, E, \vec{B})\!:~ \rho > 0,~ E - \frac{\norm{\vec{m}}{2}^2}{2\rho} - \frac{\norm{\vec{B}}{2}^2}{2} > 0\Big\}.
\end{align*}
The set $G$ is a convex set \cite{wu2025high} and numerical schemes that preserve their solutions in this set are called invariant-domain-preserving schemes. 
Preserving this invariant domain is crucial for stability, since negative density or negative internal energy lead to loss of hyperbolicity, rendering the discrete problem ill-posed \cite{cheng2013positivity,derigs2016novel,wu2018positivity}.
In this paper, we introduce a high-order accurate optimization-based limiter to enforce conservation and preserve the invariant domain.

\subsection{Several existing approaches}
The existing invariant-domain-preserving methods for ideal MHD equations can be broadly divided into explicit positivity-preserving limiters and flux-limiting approaches. We refer to \cite{wu2025high,wu2023geometric} for a comprehensive overview.
\par
One of the most popular approaches for solving MHD equations is designed for explicit time integration. 
The Zhang--Shu positivity-preserving limiter~\cite{zhang2010maximum,zhang2010positivity} for finite volume and discontinuous Galerkin (DG) methods preserves local conservation and high-order accuracy, provided that the cell averages already lie in the admissible set, which can be proven under suitable CFL for many systems such as compressible Euler equations  \cite{wu2025high}. In~\cite{cheng2013positivity}, such a positivity-preserving limiter was applied to DG schemes. Invariant-domain-preserving DG schemes on general meshes and via geometric quasilinearization were constructed in~\cite{wu2018provably,wu2019provably,wu2022provably}. A systematic positivity analysis of schemes for MHD equations was given in~\cite{wu2018positivity}, and positivity-preserving finite difference WENO schemes with constrained transport were constructed in~\cite{christlieb2015positivity}. 
In particular, it has been proven in~\cite{wu2018positivity} that positivity of the cell averages is closely related to a discrete divergence free condition. 
\par
Another class of invariant-domain-preserving methods is based on flux limiting, which originated from the flux-corrected transport method \cite{zalesak1979fully}. 
Flux limiters have been designed and applied to phase-field models \cite{liu2022pressure,frank2020bound} and porous media flow \cite{joshaghani2022maximum}.
The modern reformulation as convex limiting provides a general framework for invariant-domain preservation for compressible Euler and Navier--Stokes equations by blending a high-order scheme with a low-order invariant-domain-preserving scheme~\cite{guermond2018second,guermond2021second}.
Extensions to DG discretizations include~\cite{hajduk2021monolithic,pazner2021sparse}. A structure-preserving convex-limiting method specifically for ideal MHD was recently developed in~\cite{dao2024structure}.
These approaches rely on the availability of a suitable low-order invariant-domain-preserving scheme for the target system, which for MHD requires additional care due to the nonlinear structure of the admissible set $G$.
High-order accuracy is typically maintained when flux limiters are applied correctly. However, a rigorous proof of accuracy preservation is available only for simpler equations~\cite{xu2014parametrized}.
\par
Optimization-based limiters are attractive because conservation and admissibility can be enforced simultaneously through a constrained minimization problem.
Different optimization-based limiters were developed for scalar equations and spectral element methods in~\cite{guba2014optimization,van2019positivity,ruppenthal2023optimal}.
A systematic two-stage limiting framework using the Douglas--Rachford (DR) splitting method to solve the minimization problem was introduced in~\cite{liu2024simple} for the phase-field multi-phase flow. Asymptotic linear convergence and nearly optimal parameter selection of the splitting iteration are rigorously proved for the scalar case.
This framework is broadly applicable and has been extended to Fokker--Planck equations~\cite{LHTZ2024FP}, compressible Navier--Stokes equations~\cite{liu2024optimization}, and most recently to compressible Euler equations for vector admissible set with $L^2$ and $L^1$ cell average limiters~\cite{liu2025efficient}.

\subsection{Optimization-based cell average limiter}
Let $\overline{\vec{U}_h}$ be the cell averages of a DG solution $\vec{U}_h = (\rho_h, \vec{m}_h, E_h, \vec{B}_h)$.
We seek a piecewise constant polynomial $\vec{X}_h^\ast$ that minimizes the $L^2$ distance to $\overline{\vec{U}_h}$ subject to constraints of preserving global conservation and numerical admissibility.
We choose the $L^2$ objective function for both efficiency and accuracy; specifically, we solve
\begin{align*}
\min_{\vec{X}_h} \norm{\vec{X}_h - \overline{\vec{U}_h}}{L^2}^2 ~\text{subject to}~ \int_\Omega \vec{X}_h = \int_\Omega \vec{U}_h ~\text{and}~ \on{\vec{X}_h}{K_i}\in G^\varepsilon ~\text{for all cells}~K_i,
\end{align*}
where, for a small $\varepsilon > 0$, the numerical admissible set of the MHD system \eqref{eq:MHD_system} is
\begin{align}\label{eq:MHD_set_Geps}
G^\varepsilon = \Big\{(\rho, \vec{m}, E, \vec{B})\!:~ \rho \geq \varepsilon,~ E - \frac{\norm{\vec{m}}{2}^2}{2\rho} - \frac{\norm{\vec{B}}{2}^2}{2} \geq \varepsilon\Big\}.
\end{align} 
Assuming the cell averages of the exact solution are feasible, then the minimizer $\vec{X}_h^\ast$ satisfies
\begin{align*}
\norm{\vec{X}_h^\ast - \overline{\vec{U}^\mathrm{exact}}}{L^2} \leq \norm{\overline{\vec{U}_h} - \overline{\vec{U}^\mathrm{exact}}}{L^2}.
\end{align*}
Here, $\overline{\vec{U}^\mathrm{exact}}$ denotes the cell averages of the exact solution. This estimate gives a quasi-optimal bound for the postprocessed cell averages relative to feasible exact cell averages. See Section~\ref{sec:limiter:accuracy} for a proof.
\par
The primary computational challenge in implementing the cell average limiter is projecting the out-of-bound cell averages onto the numerical admissible set. In the scalar case, projection onto an interval $[m, M]$ is a simple clipping operation \cite{liu2024simple,LHTZ2024FP,liu2024optimization}.
For the compressible Euler equations, its convex admissible set enjoys a relatively simple structure and the projection is computed via a closed-form formula derived using the Karush--Kuhn--Tucker (KKT) conditions~\cite{liu2025efficient}.
For the MHD equations, however, the magnetic field $\vec{B}$ appears in the total pressure, coupling $\vec{B}$ nonlinearly with $(\rho, \vec{m}, E)$ through the internal energy constraint.
This coupling makes a closed-form projection formula analytically difficult.
\par
We resolve this difficulty with a slicing algorithm that decomposes the numerical admissible set into slices parameterized by $\norm{\vec{B}}{2}^2 = \beta$: for each fixed $\beta$, the slice decouples into subproblems for $\vec{B}$ and $(\rho, \vec{m}, E)$ separately, and each of them can be solved in closed form.
The projection then reduces to a one-dimensional minimization over the closed interval $I = [\beta_{\mathrm{low}}, \beta_{\mathrm{high}}]$, where $\beta_{\mathrm{low}}$ and $\beta_{\mathrm{high}}$ are given in Lemma~\ref{lem:search_interval_I}.
We prove that this reduced objective is strictly convex and continuous on $I$, implying the existence and uniqueness of the projection, and compute the minimizer efficiently using the Brent method~\cite[Section~5]{brent2013algorithms}.
\par
Once cell averages are in the admissible set, the Zhang--Shu positivity-preserving limiter~\cite{zhang2010maximum,zhang2010positivity,liu2023positivity} can be applied to further process the quadrature point values, which also preserves the order of accuracy.
\par
Our optimization-based limiter addresses admissibility restoration only and does not enforce the discrete divergence-free condition on $\vec{B}$. In our implementation, a discrete divergence-free projection, which enforces local divergence-free on each cell, is applied at the start of each time step before the limiter acts on the cell averages. The two steps are treated separately.

\subsection{Efficient splitting methods}
In large-scale high-resolution simulations, the number of cell averages to be processed can be very large. Designing an efficient algorithm for solving the minimization problem is critical to the practical success of the optimization-based approach.
\par
The DR splitting was originally introduced to solve heat equations~\cite{peaceman1955numerical,douglas1956numerical}. It was later generalized by Lions and Mercier for handling a sum of two maximal monotone operators, resulting in the generalized DR method~\cite{lions1979splitting}.
The nonsmooth convex minimization model for the scalar cell average limiter can be efficiently solved by the generalized DR method, which is equivalent to several well-known methods, including PDHG~\cite{chambolle2016introduction}, ADMM~\cite{fortin2000augmented}, and the split Bregman method~\cite{goldstein2009split}. We refer to~\cite{demanet2016eventual,zhang2026ubiquitous} and references therein for further details on these equivalences.
However, no general parameter selection strategy is available for the generalized DR method when applied to the cell average limiter with vector invariant domain.
Moreover, the objective in \eqref{eq:invariant_domain_limiter2} consists of three parts: the quadratic distance, the conservation constraint, and the admissibility constraint. This structure motivates the use of a three-operator splitting method.
\par
Three-operator splitting methods~\cite{glowinski2017splitting,ryu2022large}, such as Davis--Yin (DY) splitting~\cite{davis2017three}, three-block ADMM~\cite{lin2015global,cai2014direct}, and others~\cite{anshika2025three}, are popular and natural choices for solving the composite convex optimization of the form
\begin{align*}
\min_x d_1(x) + d_2(x) + d_3(x).
\end{align*}
Recently, the DY method has been successfully applied to optimization-based limiters with conservation and invariant-domain constraints for the compressible Euler equations~\cite{liu2025efficient}. The DY iteration is given by:
\begin{align}\label{eq:DY_algorithm}
\text{(DY)}~
\begin{cases}
x^{k+1/2} &\hspace{-0.75em}= \prox_{d_3}^\gamma(z^k),\\
x^{k+1} &\hspace{-0.75em}= \prox_{d_1}^\gamma(2x^{k+1/2} - z^{k} - \gamma\grad{d_2}(x^{k+1/2})),\\
z^{k+1} &\hspace{-0.75em}= z^{k} + x^{k+1} - x^{k+1/2}.
\end{cases}
\end{align}
Here, the proximal operator with parameter $\gamma > 0$ for the convex function $d_1$ is defined as follows:
\begin{align*}
\prox_{d_1}^\gamma(x) 
= \mathrm{argmin}_{y} d_1(y) + \frac{1}{2\gamma}\norm{y-x}{2}^2,
\end{align*}
and $\prox_{d_3}^\gamma$ is defined similarly.
For proper closed convex functions $d_1$, $d_2$, and $d_3$, where $\grad{d_2}$ is $L$-Lipschitz continuous, iteration \eqref{eq:DY_algorithm} converges for any constant step size $\gamma\in(0, 2/L)$.
When the iteration converges to machine precision, the two constraints (conservation and invariant domain) are enforced up to round-off errors.
In practice, the DY method with $\gamma = 1/L$ outperforms other alternatives~\cite{anshika2025three} and eliminates the need for parameter tuning.

\subsection{Contributions and organization of the paper}
We are not aware of any optimization-based conservative and invariant-domain-preserving limiter for the MHD equations. The nonlinear coupling of all variables $(\rho, \vec{m}, E, \vec{B})$ in the admissible set $G$ makes the projection significantly more challenging to compute than in the scalar and compressible Euler cases. Our main contribution is a slicing algorithm for computing the projection of cell averages onto the MHD numerical admissible set $G^\varepsilon$.
By decomposing $G^\varepsilon$ into slices parameterized by $\norm{\vec{B}}{2}^2 = \beta$, the projection reduces to a one-dimensional minimization over a bounded interval.
We prove that the reduced objective is strictly convex and continuous, guaranteeing the existence and uniqueness of the projection. The minimizer is computed efficiently using the Brent method.
\par
We prove that our $L^2$ optimization-based cell average limiter preserves the order of accuracy of the underlying DG scheme.
Furthermore, the global optimization problem is solved efficiently by the DY splitting method, with each iteration requiring one evaluation of the slicing algorithm and one projection onto the conservation hyperplane in closed form.
The asymptotic linear convergence is observed in our numerical tests and the DY iteration converges in several steps.
\par
We test the correctness and efficiency of the slicing algorithm for computing the projection point, and we apply the full method to the smooth circularly polarized Alfv\'en wave, Rotor problem, Orszag--Tang vortex problem, and high Mach number astrophysical jet.
The numerical results confirm the correctness, high-order accuracy, and robustness of the proposed scheme on demanding MHD problems with low density and low pressure.
\par
The rest of this paper is organized as follows.
In Section~\ref{sec:limiter}, we formulate the optimization-based cell average limiter, prove accuracy preservation, and present the DY splitting algorithm including the identification of the three operators and their proximals.
Section~\ref{sec:projection} develops the slicing algorithm for projecting onto the MHD admissible set, establishes the strict convexity and continuity of the reduced objective, and details the closed-form projection formulas for each fixed slice. 
We present the numerical benchmark tests in Section~\ref{sec:numerical_experiments}. 
The explicit projection algorithm onto the compressible Euler-like admissible set and the Brent method are given in the Appendices.

\section{Constraint optimization-based limiter}\label{sec:limiter}
\!We propose a postprocessing step that restores the invariant domain while preserving global conservation. The postprocessed piecewise constant polynomial is chosen to minimize the $L^2$ distance to the original cell averages, subject to these two constraints.
\par
Given a numerical solution $\vec{U}_h$ to a conservative DG scheme for MHD equations \eqref{eq:MHD_system}, where $\vec{U}_h = (\rho_h, \vec{m}_h, E_h, \vec{B}_h)$, let $\overline{\vec{U}_h}$ denote the cell average of $\vec{U}_h$, namely, on each cell $K$, we have $\overline{\vec{U}_h}|_{K} = \frac{1}{\abs{K}}\int_K \vec{U}_h$. 
We seek a piecewise constant polynomial $\vec{X}_h$ that minimizes the distance to $\overline{\vec{U}_h}$ under the constraints of preserving conservation and invariant domain:
\begin{align}\label{eq:introduction_opt_model}
\min_{\vec{X}_h} \norm{\vec{X}_h - \overline{\vec{U}_h}}{L^2}^2 ~\text{subject to}~
\int_\Omega \vec{X}_h = \int_\Omega \vec{U}_h ~\text{and}~ 
\on{\vec{X}_h}{K_i}\in G^\varepsilon ~\text{for all cells}~K_i,
\end{align}
where $G^\varepsilon$ is the numerical admissible set defined in \eqref{eq:MHD_set_Geps}.
For positive density, the set $G^\varepsilon$ is closed and convex, which implies that \eqref{eq:introduction_opt_model} has a unique minimizer, denoted by $\vec{X}_h^\ast$. Then, the postprocessed DG polynomial can be written as 
\begin{align*}
\widehat{\vec{U}}_h = (\vec{U}_h - \overline{\vec{U}_h}) + \vec{X}_h^\ast,
\end{align*}
and it preserves global conservation of density, momentum, total energy, and magnetic field. In addition, $\widehat{\vec{U}}_h$ has cell averages in set $G^\varepsilon$.

\subsection{Matrix-vector form and accuracy}\label{sec:limiter:accuracy}
To solve \eqref{eq:introduction_opt_model}, we introduce a matrix $\overline{\vecc{U}} \in \IR^{N\times(2+2n)}$ to store the cell averages of $\vec{U}_h$, where the $i$-th row is given by 
\begin{align*}
\begin{bmatrix}
\displaystyle \frac{1}{\abs{K_i}}\int_{K_i} \rho_h & 
\displaystyle \frac{1}{\abs{K_i}}\int_{K_i} \vec{m}_h & 
\displaystyle \frac{1}{\abs{K_i}}\int_{K_i} E_h & 
\displaystyle \frac{1}{\abs{K_i}}\int_{K_i} \vec{B}_h
\end{bmatrix}.
\end{align*}
Let $\norm{\cdot}{F}$ denote the Frobenius norm and $N$ be the total number of mesh cells. 
Define the indicator function of a set $\Lambda$ as: $\iota_\Lambda(\vecc{X}) = 0$, if $\vecc{X}\in \Lambda$, otherwise $\iota_\Lambda(\vecc{X}) = +\infty$.
The model \eqref{eq:introduction_opt_model} is equivalent to the following unconstrained minimization problem: find $\vecc{X}\in\IR^{N\times(2+2n)}$, such that
\begin{align}\label{eq:invariant_domain_limiter2}
\min_{\vecc{X}}~~& \frac{1}{2}\norm{\vecc{X} - \overline{\vecc{U}}}{F}^2 + \iota_{\Lambda_1}(\vecc{X}) + \iota_{\Lambda_2}(\vecc{X}),\nonumber\\
\text{where}~~&
\Lambda_1 = \{\vecc{X}\!:~ \vecc{A}\vecc{X} = \transpose{\vec{b}}\}\\
\text{and}~~&
\Lambda_2 = \{\vecc{X}\!:~ \text{the $i$-th row}~\vecc{X}_i \in G^\varepsilon,~ \forall i\}.\nonumber
\end{align}
For simplicity, in the rest of this paper we only consider uniform meshes. On a uniform mesh, we have $\vecc{A} = [1,1,\cdots,1]\in\IR^{1\times N}$ and $\transpose{\vec{b}} = \vecc{A}\overline{\vecc{U}}$. 
Since both sets $\Lambda_1$ and $\Lambda_2$ are convex, their corresponding indicator functions $\iota_{\Lambda_1}$ and $\iota_{\Lambda_2}$ are also convex. Therefore, \eqref{eq:invariant_domain_limiter2} defines a strongly convex minimization problem, and its solution is unique.
\par
Next, we show that the limiter does not destroy the order of accuracy. Let $\vecc{X}^\ast$ denote the solution of \eqref{eq:invariant_domain_limiter2} and let $\overline{\vecc{U}^\mathrm{exact}}\in \IR^{N\times (2+2n)}$ be a matrix that stores the cell averages of the exact solution, where the $i^\mathrm{th}$ row is given by
\begin{align*}
\begin{bmatrix}
\displaystyle \frac{1}{\abs{K_i}}\int_{K_i} \overline{\rho_h^\mathrm{exact}} & 
\displaystyle \frac{1}{\abs{K_i}}\int_{K_i} \overline{\vec{m}_h^\mathrm{exact}} & 
\displaystyle \frac{1}{\abs{K_i}}\int_{K_i} \overline{E_h^\mathrm{exact}} & 
\displaystyle \frac{1}{\abs{K_i}}\int_{K_i} \overline{\vec{B}_h^\mathrm{exact}}
\end{bmatrix}.
\end{align*}
Following a similar argument as in \cite[Theorem~1]{liu2025efficient}, the sets $\Lambda_1$ and $\Lambda_2$ are convex and closed, which gives $\Lambda_1\cap\Lambda_2$ is a convex closed set. 
Thus, $\overline{\vecc{U}^\mathrm{exact}}$ and $\vecc{X}^\ast$ belong to $\Lambda_1\cap\Lambda_2$ implies $\lambda \overline{\vecc{U}^\mathrm{exact}} + (1-\lambda)\vecc{X}^\ast \in \Lambda_1\cap\Lambda_2$, for any $\lambda\in[0,1]$. 
Define
\begin{align*}
\phi(\lambda) &= \norm{\overline{\vecc{U}} - (\lambda \overline{\vecc{U}^\mathrm{exact}} + (1-\lambda)\vecc{X}^\ast)}{F}^2 \nonumber\\
&= \lambda^2\norm{\overline{\vecc{U}^\mathrm{exact}} - \vecc{X}^\ast}{F}^2 - 2\lambda(\overline{\vecc{U}} - \vecc{X}^\ast):(\overline{\vecc{U}^\mathrm{exact}} - \vecc{X}^\ast) + \norm{\overline{\vecc{U}} - \vecc{X}^\ast}{F}^2.
\end{align*}
If $\vecc{X}^\ast \neq \overline{\vecc{U}^\mathrm{exact}}$, then $\phi(\lambda)$ is a quadratic function. From \eqref{eq:invariant_domain_limiter2}, we know $\vecc{X}^\ast$ minimizes $\norm{\overline{\vecc{U}}-\vecc{X}}{F}^2$ for all $\vecc{X} \in \Lambda_1\cap\Lambda_2$. Thus, $\phi(\lambda)$ achieves its minimum at $\lambda = 0$, which gives
\begin{align*}
\frac{(\overline{\vecc{U}} - \vecc{X}^\ast):(\overline{\vecc{U}^\mathrm{exact}} - \vecc{X}^\ast)}{\norm{\overline{\vecc{U}^\mathrm{exact}} - \vecc{X}^\ast}{F}^2} \leq 0
\quad\Rightarrow\quad
(\vecc{X}^\ast - \overline{\vecc{U}}):(\overline{\vecc{U}^\mathrm{exact}} - \vecc{X}^\ast) \geq 0.
\end{align*}
Thus, we obtain 
\begin{align*}
\norm{\overline{\vecc{U}^\mathrm{exact}}-\overline{\vecc{U}}}{F}^2
&= \norm{\overline{\vecc{U}^\mathrm{exact}}-\vecc{X}^\ast}{F}^2 + 2(\vecc{X}^\ast - \overline{\vecc{U}}):(\overline{\vecc{U}^\mathrm{exact}}- \vecc{X}^\ast) + \norm{\vecc{X}^\ast - \overline{\vecc{U}}}{F}^2 \\
&\geq \norm{\overline{\vecc{U}^\mathrm{exact}}-\vecc{X}^\ast}{F}^2.
\end{align*}
It is straightforward to verify that when $\vecc{X}^\ast = \overline{\vecc{U}^\mathrm{exact}}$, the same inequality still holds. Therefore, the postprocessed cell average satisfies the same quasi-optimal bound when the cell averages of the exact solution are in the numerical admissible set $G^\varepsilon$.

\subsection{Efficient optimization method}\label{sec:efficient_solver}
We apply the DY method \eqref{eq:DY_algorithm} to solve the minimization problem \eqref{eq:invariant_domain_limiter2}. 
\par
Partition the matrix $\overline{\vecc{U}} = [\vec{u}, \vec{v}_1, \cdots, \vec{v}_n, \vec{w}, \vec{z}_1, \cdots, \vec{z}_n]$ into $2+2n$ columns, corresponding, in order to the cell averages of density, the $n$ momentum components, the total energy, and the $n$ magnetic field components.
Denote the entries in vector $\transpose{\vec{b}}$ by $[b_\rho, b_{m_1}, \cdots, b_{m_n}, b_E, b_{B_1}, \cdots, b_{B_n}]$.
Note that $\overline{\vecc{U}}$ and $\vec{b}$ are given quantities, which are computed from the DG solution. For the unknown $\vecc{X}$ in \eqref{eq:invariant_domain_limiter2}, we similarly define the partition $\vecc{X} = [\vec{\rho}, \vec{m}_1, \cdots, \vec{m}_n, \vec{E}, \vec{B}_1, \cdots, \vec{B}_n]$.
Then, the minimization problem \eqref{eq:invariant_domain_limiter2} is equivalent to
\begin{align*}
\min_{\vecc{X}}&~ \frac{1}{2}\Big(\norm{\vec{\rho} - \vec{u}}{2}^2 + \sum_{i=1}^{n}\norm{\vec{m}_i - \vec{v}_i}{2}^2 + \norm{\vec{E} - \vec{w}}{2}^2 + \sum_{i=1}^{n}\norm{\vec{B}_i - \vec{z}_i}{2}^2\Big) + \iota_{\Lambda_1}(\vecc{X}) + \iota_{\Lambda_2}(\vecc{X}),\nonumber\\
\text{where}&~ \Lambda_1 = \{\vecc{X}\!:\,\vecc{A}\vec{\rho} = b_\rho,~ \vecc{A}\vec{m}_1 = b_{m_1},~ \cdots, \vecc{A}\vec{m}_n = b_{m_n},~\\ &\hspace{5.85cm} \vecc{A}\vec{E} = b_E,~ \vecc{A}\vec{B}_1 = b_{B_1},~ \cdots, \vecc{A}\vec{B}_n = b_{B_n}\}\\
\text{and}&~ \Lambda_2 = \{\vecc{X}\!:\,\transpose{[\rho_i, m_{1i}, \cdots, m_{ni}, E_i, B_{1i}, \cdots, B_{ni}]} \in G^\varepsilon,~ \forall i\}.\nonumber
\end{align*}
Let us split the objective function in a manner that facilitates the derivation of explicit formulas. We choose
\begin{subequations}
\begin{align}
d_1(\vecc{X}) &= \iota_{\Lambda_1}(\vec{\rho}, \vec{m}_1, \cdots, \vec{m}_n, \vec{E}, \vec{B}_1, \cdots, \vec{B}_n), \label{eq:MHD_DY_f} \\
d_2(\vecc{X}) &= \frac{1}{2}\Big(\norm{\vec{\rho} - \vec{u}}{2}^2 + \sum_{i=1}^{n}\norm{\vec{m}_i - \vec{v}_i}{2}^2 + \norm{\vec{E} - \vec{w}}{2}^2 + \sum_{i=1}^{n}\norm{\vec{B}_i - \vec{z}_i}{2}^2\Big),\label{eq:MHD_DY_h} \\
d_3(\vecc{X}) &= \iota_{\Lambda_2}(\vec{\rho}, \vec{m}_1, \cdots, \vec{m}_n, \vec{E}, \vec{B}_1, \cdots, \vec{B}_n). \label{eq:MHD_DY_g}
\end{align}
\end{subequations}
Let $\vecc{A}^+ = \transpose{\vecc{A}}(\vecc{A}\transpose{\vecc{A}})^{-1}$ denote the pseudo inverse of $\vecc{A}$. Associated to function $d_1$ in \eqref{eq:MHD_DY_f}, the proximal $\prox_{d_1}^\gamma$ maps
\begin{align*}
\vec{\rho} &\rightarrow \vecc{A}^+(b_\rho - \vecc{A}\vec{\rho}) + \vec{\rho}, & 
\vec{m}_i &\rightarrow \vecc{A}^+(b_{m_i} - \vecc{A}\vec{m}_i) + \vec{m}_i, &
\text{for}~i=1,\cdots,n,\\
\vec{E} &\rightarrow \vecc{A}^+(b_E - \vecc{A}\vec{E}) + \vec{E}, & 
\vec{B}_i &\rightarrow \vecc{A}^+(b_{B_i} - \vecc{A}\vec{B}_i) + \vec{B}_i, &
\text{for}~i=1,\cdots,n.
\end{align*}
The proximal of an indicator of a set is the Euclidean projection onto that set. 
Thus, computing the proximal for $d_3$ in \eqref{eq:MHD_DY_g} reduces to finding the projection point onto the numerical admissible set $G^\varepsilon$.

\section{Projection onto admissible set}\label{sec:projection}
In this part, we focus on computing the projection of a given point $(u, \vec{v}, w, \vec{z})$ onto the numerical admissible set $G^\varepsilon$ defined in \eqref{eq:MHD_set_Geps} for the MHD system.
The projection of a point on the nonempty closed convex set is unique. Determining this projection reduces to solving the following constraint minimization problem:
\begin{subequations}\label{eq:problem_org}
\begin{align}
\min_{\rho, \vec{m}, E, \vec{B}}&~~ \frac{1}{2}(\abs{\rho-u}^2 + \norm{\vec{m}-\vec{v}}{2}^2 + \abs{E-w}^2 + \norm{\vec{B}-\vec{z}}{2}^2) \label{eq:problem_org_1}\\
\text{subject to}&~~ \rho\geq\varepsilon ~~\text{and}~~ E - \frac{\norm{\vec{m}}{2}^2}{2\rho} - \frac{\norm{\vec{B}}{2}^2}{2} \geq \varepsilon. \label{eq:problem_org_2}
\end{align}
\end{subequations}
A natural starting point to address optimization \eqref{eq:problem_org} is to formulate the KKT conditions. The constraints \eqref{eq:problem_org_2} represent primal feasibility. The parameters $\lambda \geq 0$ and $\mu \geq 0$ represent dual feasibility. We have the stationarity conditions
\begin{subequations}\label{eq:KKT_stationarity}
\begin{align}
\rho - u - \lambda - \mu\frac{\norm{\vec{m}}{2}^2}{2\rho^2} = 0, &&
\vec{m} - \vec{v} + \mu \frac{\vec{m}}{\rho} = \vec{0}, \label{eq:KKT_stationarity_1} \\
E - w - \mu = 0, &&
\vec{B} - \vec{z} + \mu \vec{B} = \vec{0}. \label{eq:KKT_stationarity_2}
\end{align}
\end{subequations}
And the complementary slackness
\begin{align*}
\lambda(\varepsilon - \rho) = 0, &&
\mu\Big(\varepsilon - E + \frac{\norm{\vec{m}}{2}^2}{2\rho} + \frac{\norm{\vec{B}}{2}^2}{2}\Big) = 0.
\end{align*}
Notice, if $\vec{z} = \vec{0}$, then \eqref{eq:KKT_stationarity_2} gives $(1 + \mu)\vec{B} = \vec{0}$. By dual feasibility, we know $1+\mu>0$, which implies $\vec{B} = \vec{0}$.
Thus, when $\vec{z} = \vec{0}$, the minimization problem \eqref{eq:problem_org} reduces to
\begin{subequations}\label{eq:problem_org_z0}
\begin{align}
\min_{\rho, \vec{m}, E}&~~ \frac{1}{2}(\abs{\rho-u}^2 + \norm{\vec{m}-\vec{v}}{2}^2 + \abs{E-w}^2)\\
\text{subject to}&~~ \rho\geq\varepsilon ~~\text{and}~~ E - \frac{\norm{\vec{m}}{2}^2}{2\rho} \geq \varepsilon,
\end{align}
\end{subequations}
which has been solved in \cite{liu2025efficient}. 
Therefore, in the rest of this paper, we only need to focus on the case that $\vec{z} \neq \vec{0}$. 
Unlike projection onto the admissible set for Euler equations in \cite{liu2025efficient}, the KKT conditions associated to the MHD system have no tractable closed-form solution. This motivates the numerical construction below.

\subsection{Preliminary lemmas}
Let $g$ be a continuous convex function, and let $G$ be a convex set. Suppose $G$ admits a decomposition parameterized by $\beta$, namely $G = \cup_{\beta} G_\beta$, where each set $G_\beta$ is `nice' in the sense that the minimum of $g$ over $G_\beta$ exists. 
\begin{lemma}\label{lem:double_min}
Suppose that both $\displaystyle\min_{\vec{x}\in G} g(\vec{x})$ and $\displaystyle\min_\beta\min_{\vec{x}\in G_\beta}g(\vec{x})$ exist, then we have
\begin{align}\label{eq:double_min}
\min_{\vec{x}\in G} g(\vec{x}) = \min_\beta\min_{\vec{x}\in G_\beta}g(\vec{x}).
\end{align}
\end{lemma}
\begin{proof}
\emph{i}) The right-hand side of \eqref{eq:double_min} is attainable, i.e. there exists $\beta = \beta^\ast$, when $\vec{x} = \vec{x}^\ast \in G_{\beta^\ast} \subset G$, we have $\displaystyle g(\vec{x}^\ast) = \min_{\beta}\min_{\vec{x}\in G_\beta}g(\vec{x})$.
Since $\vec{x}^\ast\in G$, then $\displaystyle \min_{\vec{x}\in G}g(\vec{x}) \leq g(\vec{x}^\ast)$. We obtain $\displaystyle \min_{\vec{x}\in G}g(\vec{x}) \leq \min_\beta\min_{\vec{x}\in G_\beta}g(\vec{x})$.
\emph{ii}) The left-hand side of \eqref{eq:double_min} is attainable, i.e. there exists $\vec{x}^\ast\in G$, such that $\displaystyle g(\vec{x}^\ast) = \min_{\vec{x}\in G} g(\vec{x})$.
The $\vec{x}^\ast \in G = \cup_\beta G_\beta$ implies that there exists $\beta = \beta^\ast$, such that $\vec{x}^\ast \in G_{\beta^\ast}$. Notice that the minimum of $g$ over $G_{\beta^\ast}$ exists. We have $\displaystyle\min_{\vec{x}\in G_{\beta^\ast}} g(\vec{x}) \leq g(\vec{y})$ holds for any $\vec{y}\in G_{\beta^\ast}$.
Pick $\vec{y} = \vec{x}^\ast$, we get $\displaystyle \min_{\vec{x}\in G_{\beta^\ast}} g(\vec{x}) \leq g(\vec{x}^\ast)$.
Thus, we obtain $\displaystyle \min_{\beta} \min_{\vec{x}\in G_{\beta}} g(\vec{x}) \leq \min_{\vec{x}\in G_{\beta^\ast}} g(\vec{x}) \leq g(\vec{x}^\ast) = \min_{\vec{x}\in G} g(\vec{x})$.
Combining \emph{i}) and \emph{ii}), we conclude the proof.
\end{proof}
\begin{lemma}\label{lem:lem_decouple}
Suppose that all minima appearing below are attainable, then we have  
\begin{align}\label{eq:lem_decouple}
\min_{(\vec{x},\vec{y})\in F\times H} f(\vec{x}) + h(\vec{y}) = \min_{\vec{x}\in F} f(\vec{x}) + \min_{\vec{y}\in H} h(\vec{y}).
\end{align}
Here, $F\times H$ is the product of sets $F$ and $H$, namely $(\vec{x},\vec{y})\in F\times H \Leftrightarrow \vec{x}\in F$ and $\vec{y}\in H$.
\end{lemma}
\begin{proof}
\emph{i}) The right-hand side of \eqref{eq:lem_decouple} is attainable, i.e. there exists $\vec{x}^\ast\in F$, such that $\displaystyle f(\vec{x}^\ast) = \min_{\vec{x}\in F} f(\vec{x})$, and there exists $\vec{y}^\ast\in H$, such that $\displaystyle h(\vec{y}^\ast) = \min_{\vec{y}\in H} h(\vec{y})$.
Since $(\vec{x}^\ast, \vec{y}^\ast)\in F\times H$, we obtain
\begin{align*}
\min_{(\vec{x},\vec{y})\in F\times H} f(\vec{x}) + h(\vec{y}) \leq f(\vec{x}^\ast) + h(\vec{y}^\ast) 
= \min_{\vec{x}\in F} f(\vec{x}) + \min_{\vec{y}\in H} h(\vec{y}).
\end{align*}
\emph{ii}) The left-hand side of \eqref{eq:lem_decouple} is attainable, i.e. there exist $(\vec{x}^\ast, \vec{y}^\ast) \in F\times H$, such that $\displaystyle f(\vec{x}^\ast) + h(\vec{y}^\ast) = \min_{(\vec{x},\vec{y})\in F\times H} f(\vec{x}) + h(\vec{y})$.
Notice that $(\vec{x}^\ast, \vec{y}^\ast)\in F\times H$ implies $\vec{x}^\ast\in F$ and $\vec{y}^\ast \in H$. We have $\displaystyle \min_{\vec{x}\in F} f(\vec{x}) \leq f(\vec{x}^\ast)$ and $\displaystyle \min_{\vec{y}\in H} h(\vec{y}) \leq h(\vec{y}^\ast)$.
Thus, we obtain
\begin{align*}
\min_{\vec{x}\in F} f(\vec{x}) + \min_{\vec{y}\in H} h(\vec{y}) 
\leq f(\vec{x}^\ast) + h(\vec{y}^\ast) 
= \min_{(\vec{x},\vec{y})\in F\times H} f(\vec{x}) + h(\vec{y}).
\end{align*}
Combining \emph{i}) and \emph{ii}), we conclude the proof.
\end{proof}
Applying Lemma~\ref{lem:double_min} and Lemma~\ref{lem:lem_decouple} gives the decoupling of the minimization problem \eqref{eq:problem_org}. We first carry out the formal decoupling and then verify that the conditions of these lemmas are satisfied, namely that all relevant minima exist.

\subsection{Decoupling of magnetic variable}
Consider decomposing the admissible set $G^\varepsilon$ into the union of sets $G^\varepsilon(\norm{\vec{B}}{2}^2 = \beta)$, where parameter $\beta\geq 0$,
\begin{align*}
G^\varepsilon(\norm{\vec{B}}{2}^2 = \beta) &= \Big\{(\rho, \vec{m}, E, \vec{B})\!:~ \rho \geq \varepsilon,~ E - \frac{\norm{\vec{m}}{2}^2}{2\rho} - \frac{\norm{\vec{B}}{2}^2}{2} \geq \varepsilon,~ \norm{\vec{B}}{2}^2 = \beta \Big\}.
\end{align*}
By Lemma~\ref{lem:double_min}, the constraint minimization problem \eqref{eq:problem_org} is equivalent to the following
\begin{align}\label{eq:problem_2}
\min_{\beta\geq0} \min_{(\rho, \vec{m}, E, \vec{B}) \in G^\varepsilon(\norm{\vec{B}}{2}^2 = \beta)}~ \frac{1}{2}(\abs{\rho-u}^2 + \norm{\vec{m}-\vec{v}}{2}^2 + \abs{E-w}^2 + \norm{\vec{B}-\vec{z}}{2}^2).
\end{align}
Recall the definition of set $G^\varepsilon(\norm{\vec{B}}{2}^2 = \beta)$, the inner minimization problem above is
\begin{align*}
\min_{\rho, \vec{m}, E, \vec{B}}&~~ \frac{1}{2}(\abs{\rho-u}^2 + \norm{\vec{m}-\vec{v}}{2}^2 + \abs{E-w}^2 + \norm{\vec{B}-\vec{z}}{2}^2) \\
\text{subject to}&~~ \rho\geq\varepsilon,~~ E - \frac{\norm{\vec{m}}{2}^2}{2\rho} - \frac{\norm{\vec{B}}{2}^2}{2} \geq \varepsilon, ~~\text{and}~~ \norm{\vec{B}}{2}^2 = \beta.
\end{align*}
Substituting the third constraint $\norm{\vec{B}}{2}^2 = \beta$ into the second constraint, we obtain
the following equivalent form.
\begin{align*}
\min_{\rho, \vec{m}, E, \vec{B}}&~~ \frac{1}{2}(\abs{\rho-u}^2 + \norm{\vec{m}-\vec{v}}{2}^2 + \abs{E-w}^2) + \frac{1}{2}\norm{\vec{B}-\vec{z}}{2}^2 \\
\text{subject to}&~~ \rho\geq\varepsilon,~~ E - \frac{\norm{\vec{m}}{2}^2}{2\rho} \geq \varepsilon + \frac{\beta}{2}, ~~\text{and}~~ \norm{\vec{B}}{2}^2 = \beta.
\end{align*}
Let $f(\rho, \vec{m}, E) = \frac{1}{2}(\abs{\rho-u}^2 + \norm{\vec{m}-\vec{v}}{2}^2 + \abs{E-w}^2)$ and $h(\vec{B}) = \frac{1}{2}\norm{\vec{B}-\vec{z}}{2}^2$. Define sets
\begin{align*}
F^\varepsilon_\beta = \Big\{(\rho, \vec{m}, E)\!:~ \rho \geq \varepsilon,~ E - \frac{\norm{\vec{m}}{2}^2}{2\rho} \geq \varepsilon + \frac{\beta}{2}\Big\}
\quad\text{and}\quad
H_\beta = \{\vec{B}\!:~ \norm{\vec{B}}{2}^2 = \beta\}.
\end{align*}
Then, the set $G^\varepsilon(\norm{\vec{B}}{2}^2 = \beta) = F^\varepsilon_\beta \times H_\beta$.
Applying Lemma~\ref{lem:lem_decouple}, we obtain
\begin{multline}\label{eq:decouple_min}
\min_{(\rho, \vec{m}, E, \vec{B}) \in G^\varepsilon(\norm{\vec{B}}{2}^2 = \beta)}~ \frac{1}{2}(\abs{\rho-u}^2 + \norm{\vec{m}-\vec{v}}{2}^2 + \abs{E-w}^2) + \frac{1}{2}\norm{\vec{B}-\vec{z}}{2}^2 \\
= \min_{(\rho, \vec{m}, E) \in F^\varepsilon_\beta}~ \frac{1}{2}(\abs{\rho - u}^2 + \norm{\vec{m} - \vec{v}}{2}^2 + \abs{E - w}^2) + \min_{\vec{B}\in H_\beta} \frac{1}{2}\norm{\vec{B}-\vec{z}}{2}^2. 
\end{multline}
The right-hand side of \eqref{eq:decouple_min} is well defined. The first `min' is taking minimum of a strongly convex function over a nonempty convex set, which has a unique minimizer. The second `min' is taking minimum of a continuous function over a bounded closed set. The explicit form of its unique minimizer is derived in Lemma~\ref{lem:magnetic_subproblem}.
To see the left-hand side of \eqref{eq:decouple_min} is well defined, we have the following result.
\begin{lemma}\label{lem:exist_min_G_beta}
A minimizer of the following function over set $G^\varepsilon(\norm{\vec{B}}{2}^2 = \beta)$ exists
\begin{align*}
g(\rho, \vec{m}, E, \vec{B}) = \abs{\rho-u}^2 + \norm{\vec{m}-\vec{v}}{2}^2 + \abs{E-w}^2 + \norm{\vec{B}-\vec{z}}{2}^2
\end{align*}
\end{lemma}
\begin{proof}
The function $g$ is continuous and the following set is closed and bounded. ($\vec{e}_1$ denotes the first standard basis in $\IR^n$)
\begin{multline*}
A = \{(\rho, \vec{m}, E, \vec{B})\!:~ \rho \geq \varepsilon,~ E - \frac{\norm{\vec{m}}{2}^2}{2\rho} \geq \varepsilon + \frac{\beta}{2},~ \norm{\vec{B}}{2}^2 = \beta~ \text{and} \\
g(\rho, \vec{m}, E, \vec{B}) \leq \abs{\varepsilon-u}^2 + \norm{\vec{v}}{2}^2 + \abs{\varepsilon+\frac{\beta}{2}-w}^2 + \norm{\sqrt{\beta}\vec{e}_1-\vec{z}}{2}^2 + 1\}.
\end{multline*}
The set $A \subset G^\varepsilon(\norm{\vec{B}}{2}^2 = \beta)$ is nonempty, since $\vec{\eta}_0 = (\varepsilon, \vec{0}, \varepsilon+\beta/2, \sqrt{\beta}\vec{e}_1) \in A$.
A continuous function on a nonempty bounded closed set has minimum, then a minimizer of $g$ on set $A$ exists, denoted by $\vec{\xi}^\ast\in A$. We have $g(\vec{\xi}^\ast) \leq g(\vec{\eta}_0)$. 
For any point $\vec{\eta}\in G^\varepsilon(\norm{\vec{B}}{2}^2 = \beta)\setminus A$, we have 
\begin{align*}
g(\vec{\eta}) &> \abs{\varepsilon-u}^2 + \norm{\vec{v}}{2}^2 + \abs{\varepsilon+\frac{\beta}{2}-w}^2 + \norm{\sqrt{\beta}\vec{e}_1-\vec{z}}{2}^2 + 1  
> g(\vec{\eta}_0) \geq g(\vec{\xi}^\ast).
\end{align*} 
Thus, there exists a minimizer of $g$ on set $G^\varepsilon(\norm{\vec{B}}{2}^2 = \beta)$.
\end{proof}
Thus, we conclude that \eqref{eq:decouple_min} is well defined.
Next, let us show the outer minimization problem in \eqref{eq:problem_2} is attainable. Moreover, this gives us an interval to search for $\beta$.

\subsection{Properties of decoupled subproblems}\label{sec:projection:properties}
Given a point $(u, \vec{v}, w, \vec{z})$ with $\vec{z}\neq\vec{0}$, let $\rho(\beta)$, $\vec{m}(\beta)$, and $E(\beta)$ denote the solution of the minimization problem 
\begin{align}\label{eq:decoupled_min_fluid}
\min_{(\rho, \vec{m}, E) \in F^\varepsilon_\beta} \frac{1}{2}(\abs{\rho - u}^2 + \norm{\vec{m} - \vec{v}}{2}^2 + \abs{E - w}^2).
\end{align}
Let $\vec{B}(\beta)$ denote the solution of the minimization problem ${\displaystyle \min_{\vec{B}\in H_\beta}} \frac{1}{2}\norm{\vec{B}-\vec{z}}{2}^2$. We have
\begin{lemma}\label{lem:magnetic_subproblem}
For $\vec{z} \neq \vec{0}$, the solution of $\min \norm{\vec{B} - \vec{z}}{2}^2$ with constraint $\norm{\vec{B}}{2}^2 = \beta$ is
\begin{align*}
\vec{B}(\beta) = \sqrt{\beta}\frac{\vec{z}}{\norm{\vec{z}}{2}}.
\end{align*}
\end{lemma}
\begin{proof}
We utilize Lagrange multiplier to solve this minimization problem. Define
\begin{align*}
\mathcal{L} &= \norm{\vec{B} - \vec{z}}{2}^2 + \lambda(\norm{\vec{B}}{2}^2 - \beta)
= (1+\lambda)\transpose{\vec{B}}\vec{B} - 2\transpose{\vec{B}}\vec{z} + \transpose{\vec{z}}\vec{z} - \lambda\beta.
\end{align*}
Taking derivatives with respect to $\vec{B}$ and $\lambda$, we obtain
\begin{subequations}
\begin{align}
\frac{\partial \mathcal{L}}{\partial \vec{B}} &= 2(1+\lambda)\vec{B} - 2\vec{z} = \vec{0} \quad\Rightarrow\quad \vec{B} = \frac{\vec{z}}{1+\lambda},\label{eq:lagrange_multiplier_1}\\
\frac{\partial \mathcal{L}}{\partial \lambda} &= \transpose{\vec{B}}\vec{B} - \beta = 0 \hspace{1.3cm}\Rightarrow\quad \norm{\vec{B}}{2}^2 = \beta.\label{eq:lagrange_multiplier_2}
\end{align}
\end{subequations}
Substituting \eqref{eq:lagrange_multiplier_1} into \eqref{eq:lagrange_multiplier_2} yields $\norm{\vec{z}}{2}^2 = (1+\lambda)^2\beta$, which implies $\vec{B}_1 = \sqrt{\beta}\vec{z}/\norm{\vec{z}}{2}$ and $\vec{B}_2 = -\sqrt{\beta}\vec{z}/\norm{\vec{z}}{2}$.
Examining their distances to $\vec{z}$, we have 
\begin{align*}
\norm{\vec{B}_1-\vec{z}}{2} 
= \left|\frac{1}{\norm{\vec{z}}{2}}\sqrt{\beta} - 1\right|\,\norm{\vec{z}}{2}
\leq \left|\frac{1}{\norm{\vec{z}}{2}}\sqrt{\beta} + 1\right|\,\norm{\vec{z}}{2} 
= \norm{\vec{B}_2-\vec{z}}{2}.
\end{align*}
Thus, $\vec{B}_1$ is the minimizer. We conclude the proof.
\end{proof}
\par
To verify the condition in Lemma~\ref{lem:double_min} holds, it is sufficient to prove that $d^2(\beta)$, as defined below, has a unique minimizer.
\begin{align}\label{eq:MHD_slicing_d2}
d^2(\beta) = \abs{\rho(\beta) - u}^2 + \norm{\vec{m}(\beta) - \vec{v}}{2}^2 + \abs{E(\beta) - w}^2 + \norm{\vec{B}(\beta) - \vec{z}}{2}^2.
\end{align}
Let us show $d^2(\beta)$ is continuous and strictly convex on $[0,+\infty)$. By Lemma~\ref{lem:magnetic_subproblem}, we rewrite $d^2(\beta) = f(\beta) + h(\beta)$, where
\begin{align}\label{eq:def_f_and_h}
f(\beta) &= \abs{\rho(\beta) - u}^2 + \norm{\vec{m}(\beta) - \vec{v}}{2}^2 + \abs{E(\beta) - w}^2 
\quad\text{and}\quad 
h(\beta) = (\sqrt{\beta} - \norm{\vec{z}}{2})^2.
\end{align}
$h(\beta)$ is continuous and strictly convex on $[0,+\infty)$. For any $\beta\in(0,+\infty)$, we have
\begin{align*}
h'(\beta) 
= 1 - \frac{\norm{\vec{z}}{2}}{\sqrt{\beta}}
\quad\text{and}\quad
h''(\beta) &= \frac{1}{2\beta\sqrt{\beta}}\,\norm{\vec{z}}{2} > 0.
\end{align*}
Thus, to prove $d^2(\beta)$ is strictly convex on $[0,+\infty)$, we only need to show $f(\beta)$ is convex.
\begin{lemma}\label{lem:points_in_Geps}
For any $\lambda\in [0,1]$, let $\beta_c = \lambda\beta_1 + (1-\lambda)\beta_2$ be a convex combination of $\beta_1$ and $\beta_2$, and let $\vec{V}(\beta) = (\rho(\beta), \vec{m}(\beta), E(\beta))$. Then, both 
\begin{align*}
(\lambda\vec{V}(\beta_1) + (1-\lambda)\vec{V}(\beta_2), \vec{B}(\beta_c))
\quad\text{and}\quad 
(\vec{V}(\beta_c), \vec{B}(\beta_c))
\end{align*}
belong to the set $G^\varepsilon(\norm{\vec{B}}{2}^2 = \beta_c)$.
\end{lemma}
\begin{proof}
From Lemma~\ref{lem:exist_min_G_beta}, we know that both $(\vec{V}(\beta_1), \vec{B}(\beta_1)) \in G^\varepsilon(\norm{\vec{B}}{2}^2 = \beta_1)$ and $(\vec{V}(\beta_2), \vec{B}(\beta_2)) \in G^\varepsilon(\norm{\vec{B}}{2}^2 = \beta_2)$ exist and satisfy
\begin{align*}
&\rho(\beta_1) \geq \varepsilon \quad\text{and}\quad
E(\beta_1) - \frac{\norm{\vec{m}(\beta_1)}{2}^2}{2\rho(\beta_1)} \geq \varepsilon + \frac{\beta_1}{2},\\
&\rho(\beta_2) \geq \varepsilon \quad\text{and}\quad
E(\beta_2) - \frac{\norm{\vec{m}(\beta_2)}{2}^2}{2\rho(\beta_2)} \geq \varepsilon + \frac{\beta_2}{2}.
\end{align*}
For any $\lambda\in[0,1]$, we have $\lambda\rho(\beta_1) + (1-\lambda)\rho(\beta_2) \geq \varepsilon > 0$.
When $\rho > 0$, it is easy to verify that the Hessian of the function $q(\vec{V}) = E - \frac{\norm{\vec{m}}{2}^2}{2\rho}$ has nonpositive eigenvalues and hence $q$ is concave \cite{zhang2017positivity}. 
By Jensen's inequality, we get
\begin{align*}
q(\lambda\vec{V}(\beta_1) + (1-\lambda)\vec{V}(\beta_2)) 
\geq \lambda q(\vec{V}(\beta_1)) + (1-\lambda) q(\vec{V}(\beta_2))
= \varepsilon + \frac{\beta_c}{2}.
\end{align*}
Thus, $(\lambda\vec{V}(\beta_1) + (1-\lambda)\vec{V}(\beta_2), \vec{B}(\beta_c))$ belongs to the set $G^\varepsilon(\norm{\vec{B}}{2}^2 = \beta_c)$.
In addition, recall that $\vec{V}(\beta_c)$ solves \eqref{eq:decoupled_min_fluid} with $\beta = \beta_c$ and $\vec{B}(\beta_c) = \sqrt{\beta_c}\vec{z}/\norm{\vec{z}}{2}$. We have
\begin{align*}
\rho(\beta_c) \geq \varepsilon, \quad
E(\beta_c) - \frac{\norm{\vec{m}(\beta_c)}{2}^2}{2\rho(\beta_c)} \geq \varepsilon + \frac{\beta_c}{2}, \quad\text{and}\quad
\norm{\vec{B}(\beta_c)}{2}^2 = \beta_c.
\end{align*}
Thus, $(\vec{V}(\beta_c), \vec{B}(\beta_c))$ also belongs to the set $G^\varepsilon(\norm{\vec{B}}{2}^2 = \beta_c)$.
\end{proof}
For any given point $(u, \vec{v}, w, \vec{z})$ with $\vec{z}\neq\vec{0}$, noticing that $\vec{V}(\beta_c)$ is the minimizer of \eqref{eq:decoupled_min_fluid} when $\beta = \beta_c$, using Lemma~\ref{lem:points_in_Geps}, we have
\begin{align*}
f(\beta_c) 
&= \norm{\vec{V}(\beta_c) - (u,\vec{v},w)}{2}^2
\leq \norm{\lambda\vec{V}(\beta_1) + (1-\lambda)\vec{V}(\beta_2) - (u,\vec{v},w)}{2}^2\\
&= \norm{\lambda\vec{V}(\beta_1) - \lambda(u,\vec{v},w) + (1-\lambda)\vec{V}(\beta_2) - (1-\lambda)(u,\vec{v},w)}{2}^2.
\end{align*}
Since a quadratic function is convex, by Jensen's inequality, for any $\lambda\in[0,1]$, we get
\begin{align*}
f(\lambda\beta_1 + (1-\lambda)\beta_2) 
\leq \lambda f(\beta_1) + (1-\lambda)f(\beta_2). 
\end{align*}
Therefore, $f(\beta)$ is convex. Recall $h(\beta)$ is strictly convex. We get $d^2(\beta) = f(\beta) + h(\beta)$ is strictly convex on $[0,+\infty)$.
Next, we show the continuity of $d^2(\beta)$.
\begin{lemma}\label{lem:continuity_d2}
The mapping $\beta \mapsto d^2(\beta)$ is continuous on the interval $[0,+\infty)$. 
\end{lemma}
\begin{proof}
The strict convexity of $d^2(\beta)$ on interval $[0,+\infty)$ ensures its continuity on $(0, +\infty)$, we only need to show it is continuous at $0$. 
The function $f(\beta)$ is well-defined for any $\beta\geq 0$, as \eqref{eq:decoupled_min_fluid} is minimizing a strongly convex function on a nonempty convex set, which guarantees a unique solution $(\rho(\beta), \vec{m}(\beta), E(\beta))$.
\par
\emph{i}) For any $\beta_2 \geq \beta_1$, we have $F^\varepsilon_{\beta_2} \subset F^\varepsilon_{\beta_1}$, which implies $f$ is monotonically increasing on $[0,+\infty)$. Specifically, for any $\beta\geq 0$, we have $F^\varepsilon_{\beta} \subset F^\varepsilon_0$, namely $f(\beta)\geq f(0)$ holds. Thus, $f$ is lower semi-continuous at $\beta=0$, for any sequence $\beta_k\rightarrow 0^+$, we have 
\begin{align*}
\liminf_{\beta_k\rightarrow 0^+} f(\beta_k) \geq f(0).
\end{align*}
\par
\emph{ii}) From $(\rho(0), \vec{m}(0), E(0))$ solving \eqref{eq:decoupled_min_fluid} with $\beta=0$, we get $(\rho(0), \vec{m}(0), E(0) + \beta/2)$ belonging to $F^\varepsilon_\beta$. This is easy to verify, since
\begin{align*}
\rho(0) \geq \varepsilon \quad\text{and}\quad
\Big(E(0) + \frac{\beta}{2}\Big) - \frac{\norm{\vec{m}(0)}{2}^2}{2\rho(0)} = \Big(E(0) - \frac{\norm{\vec{m}(0)}{2}^2}{2\rho(0)}\Big) + \frac{\beta}{2} \geq \varepsilon + \frac{\beta}{2}.
\end{align*}
Thus, we obtain
\begin{align*}
f(\beta) \leq \abs{\rho(0) - u}^2 + \norm{\vec{m}(0) - \vec{v}}{2}^2 + \abs{E(0) + \frac{\beta}{2} - w}^2
= f(0) + \beta(E(0)-w) + \frac{\beta^2}{4}.
\end{align*}
Thus, $f$ is upper semi-continuous at $\beta = 0$, for any sequence $\beta_k\rightarrow 0^+$, we have 
\begin{align*}
\limsup_{\beta_k\rightarrow 0^+} f(\beta) 
\leq \limsup_{\beta_k\rightarrow 0^+} \Big(f(0) + \beta_k(E(0)-w) + \frac{\beta_k^2}{4}\Big)
\leq f(0).
\end{align*}
Combining \emph{i}) and \emph{ii}), we conclude the proof.
\end{proof}
We next show that $d^2(\beta)$ has a unique minimizer on $[0,+\infty)$. 
By Lemma~\ref{lem:continuity_d2}, we know $d^2(\beta)$ is continuous on $[0,+\infty)$. In addition, we have $d^2(\beta)\rightarrow +\infty$ as $\beta\rightarrow +\infty$, since
\begin{align*}
d^2(\beta) = \abs{\rho(\beta) - u}^2 + \norm{\vec{m}(\beta) - \vec{v}}{2}^2 + \abs{E(\beta) - w}^2 + (\sqrt{\beta} - \norm{\vec{z}}{2})^2 \geq (\sqrt{\beta} - \norm{\vec{z}}{2})^2.
\end{align*}
There exists a sufficiently large $M > 0$, when $\beta > M$, we have $d^2(\beta) > d^2(0)$.
On the closed interval $[0, M]$, since $d^2(\beta)$ is strictly convex, there exists a unique minimizer $\beta^\ast$, namely for all $\beta\in[0, M]$, $d^2(\beta) \geq d^2(\beta^\ast)$.
In particular, $d^2(0) \geq d^2(\beta^\ast)$ holds. For any $\beta > M$, $d^2(\beta) > d^2(0)$ gives $ d^2(\beta) > d^2(\beta^\ast)$. Thus, $\beta^\ast$ is the unique minimizer.
\par
For convenience of finding the minimizer numerically, the following lemma confines our search to a finite interval.
\begin{lemma}\label{lem:search_interval_I}
The unique minimizer of $d^2(\beta)$ lies in a closed interval with upper bound $\beta_\mathrm{high} = \norm{\vec{z}}{2}^2$ and lower bound 
\begin{align}\label{eq:lower_bound_d2}
\beta_\mathrm{low} = \left(\frac{\norm{\vec{z}}{2}}{1 + \sqrt{f(0)} + \frac{\norm{\vec{z}}{2}^2}{2}}\right)^2.
\end{align}
\end{lemma}
\begin{proof}
The upper bound is easy to justify, since $d^2(\beta) = f(\beta) + h(\beta)$, where $f(\beta)$ is continuous, increasing, and convex, and $h(\beta)$ is continuous, strictly convex, and attains its minimum at $\norm{\vec{z}}{2}^2$. It remains only to derive a lower bound.
\par 
For any given point $(u,\vec{v},w)$ and $0 \leq \beta_1 < \beta_2 \leq \norm{\vec{z}}{2}^2$, recall the definition of the function $f$ in \eqref{eq:def_f_and_h}, we have
\begin{align*}
\abs{\sqrt{f(\beta_1)} - \sqrt{f(\beta_2)}} 
= \abs{\mathrm{dist}((u,\vec{v},w), F^{\varepsilon}_{\beta_1}) - \mathrm{dist}((u,\vec{v},w),F^{\varepsilon}_{\beta_2})}
\leq \mathrm{dist}_\mathrm{H}(F^{\varepsilon}_{\beta_1},F^{\varepsilon}_{\beta_2}),
\end{align*}
where the Hausdorff distance between two sets $F^{\varepsilon}_{\beta_1}$ and $F^{\varepsilon}_{\beta_2}$ is defined as follows
\begin{align*}
\mathrm{dist}_\mathrm{H}(F^{\varepsilon}_{\beta_1}, F^{\varepsilon}_{\beta_2}) 
= \max\Big\{\sup_{\vec{\xi}\in F^{\varepsilon}_{\beta_1}}\mathrm{dist}(\vec{\xi}, F^{\varepsilon}_{\beta_2}),~ \sup_{\vec{\eta}\in F^{\varepsilon}_{\beta_2}}\mathrm{dist}(F^{\varepsilon}_{\beta_1}, \vec{\eta})\Big\}.
\end{align*}
Notice, for any $\vec{\xi} \in F^{\varepsilon}_{\beta_1}$, there exists an $\vec{\eta} = \vec{\xi} + (0, \vec{0}, \frac{1}{2}(\beta_2-\beta_1)) \in F^{\varepsilon}_{\beta_2}$. The distance between this pair of points is $\frac{1}{2}\abs{\beta_2-\beta_1}$. Thus, we have
\begin{subequations}
\begin{align}\label{eq:Hausdorff_1}
\sup_{\vec{\xi}\in F^{\varepsilon}_{\beta_1}} \mathrm{dist}(\vec{\xi}, F^{\varepsilon}_{\beta_2}) \leq \frac{1}{2}\abs{\beta_2-\beta_1}.
\end{align}
Similarly, for any $\vec{\eta} \in F^{\varepsilon}_{\beta_2}$, there exists a $\vec{\xi} = \vec{\eta} - (0, \vec{0}, \frac{1}{2}(\beta_2-\beta_1)) \in F^{\varepsilon}_{\beta_1}$. The distance between this pair of points is $\frac{1}{2}\abs{\beta_2-\beta_1}$. Thus, we have
\begin{align}\label{eq:Hausdorff_2}
\sup_{\vec{\eta}\in F^{\varepsilon}_{\beta_2}} \mathrm{dist}(F^{\varepsilon}_{\beta_1}, \vec{\eta}) \leq \frac{1}{2}\abs{\beta_2-\beta_1}.
\end{align}
\end{subequations}
Combining \eqref{eq:Hausdorff_1} and \eqref{eq:Hausdorff_2}, we obtain $\abs{\sqrt{f(\beta_1)} - \sqrt{f(\beta_2)}} \leq \frac{1}{2}\abs{\beta_2 - \beta_1}$ holds for any $\beta_1$ and $\beta_2 \geq 0$.
Specifically, since $f(\beta)$ is increasing, we get for any $\beta\geq 0$, 
\begin{align*}
\sqrt{f(\beta)} \leq \sqrt{f(0)} + \frac{1}{2}\beta.
\end{align*}
Applying the inequality above with $\beta = \beta_1$ and $\beta = \beta_2$, respectively, and combining the two resulting bounds, we obtain
\begin{align*}
\sqrt{f(\beta_1)} + \sqrt{f(\beta_2)} \leq 2\sqrt{f(0)} + \frac{1}{2}(\beta_1 + \beta_2) \leq 2\sqrt{f(0)} + \norm{\vec{z}}{2}^2.
\end{align*}
Therefore, we obtain $f(\beta)$ is Lipschitz continuous on the interval $[0, \norm{\vec{z}}{2}^2]$, as for any $\beta_1,\beta_2\in [0, \norm{\vec{z}}{2}^2]$, we have
\begin{align*}
\abs{f(\beta_1) - f(\beta_2)} 
&= \big(\sqrt{f(\beta_1)} + \sqrt{f(\beta_2)}\big)\,\! \abs{\sqrt{f(\beta_1)} - \sqrt{f(\beta_2)}}\\
&\leq (\sqrt{f(0)} + \frac{1}{2}\norm{\vec{z}}{2}^2)\, \abs{\beta_1 - \beta_2}.
\end{align*}
We derive the lower bound \eqref{eq:lower_bound_d2} from a subgradient estimate. Let $\partial f$ denote the subgradient of $f$. For any $\beta^+ > \beta\geq0$, we know $f(\beta^+) \geq f(\beta) + \partial{f}(\beta)(\beta^+ - \beta)$. Recall $f$ is increasing and Lipschitz continuous, we have 
\begin{align*}
\partial{f}(\beta)(\beta^+ - \beta) \leq f(\beta^+) - f(\beta) \leq (\sqrt{f(0)} + \frac{1}{2}\norm{\vec{z}}{2}^2)\, (\beta^+ - \beta).
\end{align*}
Thus, we get $\partial f \leq \sqrt{f(0)} + \frac{1}{2}\norm{\vec{z}}{2}^2$ on $[0, \norm{\vec{z}}{2}^2]$.
Recall that $d^2(\beta)$ is strictly convex, then 
\begin{align*}
0 \in \partial d^2(\beta^\ast) = \partial f(\beta^\ast) + 1 - \frac{\norm{\vec{z}}{2}}{\sqrt{\beta^\ast}} 
\quad\Rightarrow\quad 
\frac{\norm{\vec{z}}{2}}{\sqrt{\beta^\ast}} - 1 \leq \sqrt{f(0)} + \frac{1}{2}\norm{\vec{z}}{2}^2.
\end{align*}
After solving $\beta^\ast$ from above, we obtain the lower bound \eqref{eq:lower_bound_d2}.
\end{proof}

\subsection{Slicing algorithm}
We are ready to design a scheme to compute the projection of a given point $(u, \vec{v}, w, \vec{z})$ onto the numerical admissible set $G^\varepsilon$ of the MHD system. 
\par 
If $\vec{z} = \vec{0}$, let $\rho(0), \vec{m}(0), E(0)$ be the solution of the minimization problem \eqref{eq:problem_org_z0}, whose closed-form expression is derived from KKT conditions in \cite{liu2025efficient}. In this case, $(\rho(0), \vec{m}(0), E(0), \vec{0})$ is the projection point. 
Otherwise, for $\vec{z} \neq \vec{0}$, define a closed search interval $I = [\beta_\mathrm{low}, \norm{\vec{z}}{2}^2]$, where the lower bound
\begin{align*}
\beta_\mathrm{low} = \left(\frac{\norm{\vec{z}}{2}}{1 + \sqrt{f(0)} + \frac{\norm{\vec{z}}{2}^2}{2}}\right)^2 ~\text{and}~ f(0) = \abs{\rho(0) - u}^2 + \norm{\vec{m}(0) - \vec{v}}{2}^2 + \abs{E(0) - w}^2.
\end{align*}
For $\beta \in I$, let $\rho(\beta), \vec{m}(\beta), E(\beta)$ be the solution to the following minimization problem:
\begin{align*}
\min_{\rho, \vec{m}, E}&~~ \abs{\rho-u}^2 + \norm{\vec{m}-\vec{v}}{2}^2 + \abs{E-w}^2 \nonumber\\
\text{subject to}&~~ \rho\geq\varepsilon ~~\text{and}~~ E - \frac{\norm{\vec{m}}{2}^2}{2\rho} \geq \varepsilon + \frac{\beta}{2},
\end{align*}
whose exact formulation can be obtained via KKT conditions, see Appendix~\ref{sec:appendix_projection_F_beta}.
Solve the single-variable minimization problem $\beta^\ast = \mathrm{argmin}\, d^2(\beta)$ on $I$, where
\begin{align*}
d^2(\beta) = \abs{\rho(\beta) - u}^2 + \norm{\vec{m}(\beta) - \vec{v}}{2}^2 + \abs{E(\beta) - w}^2 + (\sqrt{\beta} - \norm{\vec{z}}{2})^2.
\end{align*}
Notice $d^2(\beta) = f(\beta) + h(\beta)$, where $f(\beta)$ is continuous, increasing, and convex, and $h(\beta)$ is continuous and strictly convex. 
We can utilize the Brent method, see Appendix~\ref{sec:appendix_brent}, to compute the minimizer $\beta^\ast$ in the interval $I$ efficiently.
Then the projection point is $(\rho(\beta^\ast), \vec{m}(\beta^\ast), E(\beta^\ast), \sqrt{\beta^\ast}\vec{z}/\norm{\vec{z}}{2})$.

\begin{remark}
In multi-dimensional spaces, directly solving KKT conditions to derive the projection point for the MHD system is highly complicated, rendering it impractical in practice. 
This procedure involves solving high order polynomial equations of degree greater than four, for which no closed-form formulation exists, and therefore requires numerical root-finding. 
Although computing the eigenvalues of the companion matrix provides an efficient numerical approach for root-finding, roundoff errors can introduce spurious complex roots with small imaginary parts, which further significantly harm accuracy.
In contrast, our method operates only with real numbers, ensuring both numerical stability and guaranteed accuracy.  
\end{remark}

\section{Numerical experiments}\label{sec:numerical_experiments}
In this section, we first validate the slicing algorithm for computing the projection onto the numerical admissible set. We then present a convergence study on circularly polarized Alfv\'en waves. The rotor problem, Orszag--Tang problem, and high Mach number astrophysical jet demonstrate the accuracy and robustness of the DG scheme equipped with our limiter.
\par
In all simulations, we use $\IP^2$ DG discretization with the local Lax--Friedrichs flux, which is defined as follows:
\begin{align*}
\widehat{\vec{F}^a\cdot\normal_K}(\vec{U}^-, \vec{U}^+) = \frac{\vec{F}^a(\vec{U}^-) + \vec{F}^a(\vec{U}^+)}{2}\cdot\normal_K - \frac{\alpha_e}{2}(\vec{U}^+ - \vec{U}^-).
\end{align*}
Here, $\vec{F}^a$ represents the flux for all variables in MHD equations \eqref{eq:MHD_system}. The traces of $\vec{U}$ on the face of a cell from the interior and exterior are denoted by $\vec{U}^-$ and $\vec{U}^+$, respectively. The outward unit normal to cell $K$ is denoted by $\normal_K$, and $\alpha_e$ is the local fast magnetosonic wave speed on cell interface $e$.
We employ the three-stage, third-order SSP Runge--Kutta (RK) method for time integration. The time step is
\begin{align*}
\Delta t = \mathrm{CFL}\,\frac{\Delta x}{\max_e \alpha_e},
\end{align*}
where the CFL coefficient is $0.2$ unless otherwise specified. 
This base discretization does not preserve the invariant domain. The time step is deliberately chosen to stress the scheme so that cell averages can leave $G^\varepsilon$, thereby triggering the limiter.
\par
At the start of each time step, a discrete divergence-free projection is applied. After each RK stage, all cell averages are checked. If any violate $G^\varepsilon$, the optimization-based cell average limiter from Section~\ref{sec:limiter} is applied. The TVB limiter and the Zhang--Shu positivity-preserving limiter~\cite{zhang2017positivity} are then applied in sequence. The TVB parameter of value $100$ is used unless otherwise specified.

\subsection{Validation of the slicing algorithm}
We validate the slicing algorithm for computing the projection onto the numerical admissible set $G^\varepsilon$ of the MHD system on two test cases: a manufactured out-of-bound point and a representative out-of-bound cell average extracted from the astrophysical jet simulation in Section~\ref{sec:numerical_experiments:astro_jet}.
\par
For the manufactured test point, we set $\varepsilon = 10^{-13}$ and apply the algorithm to the out-of-bound point $[u,\vec{v},w,\vec{z}] = [1, 1.25, 2, 0, 2, 5, 1.7, 0]$, designed so that the minimizer $\beta^\ast$ lies well inside the search interval $I = [\beta_\mathrm{low}, \norm{\vec{z}}{2}^2]$.
The left of Figure~\ref{fig:slicing_validation} shows $d^2(\beta)$, $f(\beta)$, and $h(\beta)$ over $I = [0.121, 27.89]$.
The function $d^2(\beta)$ is continuous and strictly convex with a unique minimizer $\beta^\ast \approx 5.44$ lying well inside $I$, and $f(\beta)$ is increasing and convex, consistent with our theory.
\par
For the astro jet out-of-bound point, we apply the algorithm to a representative out-of-bound cell average produced during the astrophysical jet simulation of case $B_0 = \sqrt{2000}$ and $\varepsilon = 10^{-6}$. To be more specific, we choose the point
\begin{align*}
[u,\vec{v},w,\vec{z}] \approx [1.41, 1124, -0.31, 0, 4.51\times10^5, 44.9, -0.026, 0].
\end{align*}
The right panel of Figure~\ref{fig:slicing_validation} shows $d^2(\beta)$, $f(\beta)$, and $h(\beta)$ over the search interval $I \approx [1.98\times10^{-3}, 2017.5]$.
The search interval is wide, reflecting that the theoretical lower bound $\beta_\mathrm{low}$ is not a sharp estimate in this case.
The minimizer $\beta^\ast \approx 2017.5$ is located near the upper boundary of $I$. The strong magnetic field dominates and the cell average violates the admissible set only slightly with negative pressure, so the optimal projection barely adjusts the magnetic energy.
This behavior is physically natural and expected in high-Mach, strong-field simulations.
\begin{figure}[ht!]
\begin{center}
\begin{tabularx}{\linewidth}{@{}c@{~~}c@{}}
\includegraphics[width=0.485\textwidth]{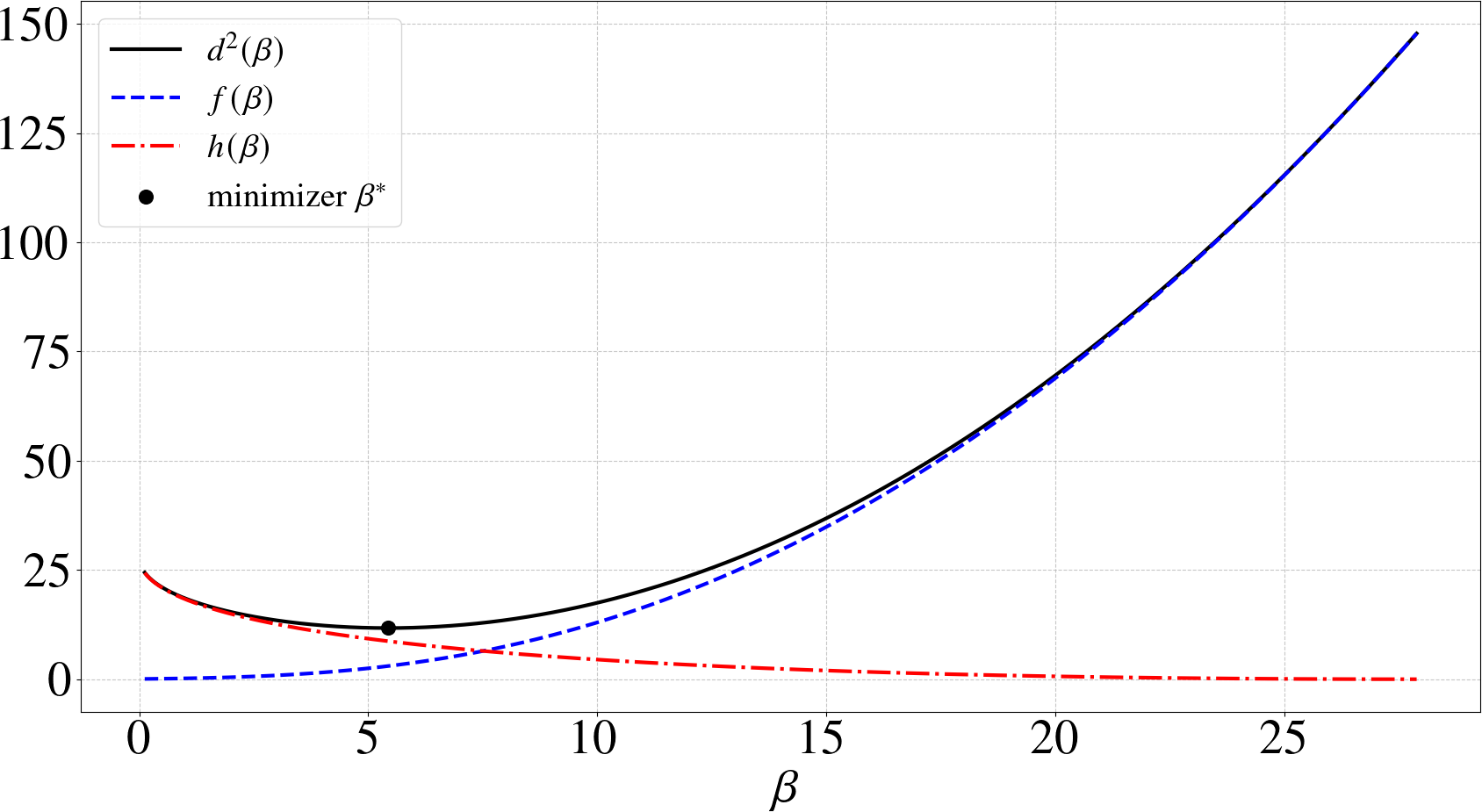} &
\includegraphics[width=0.485\textwidth]{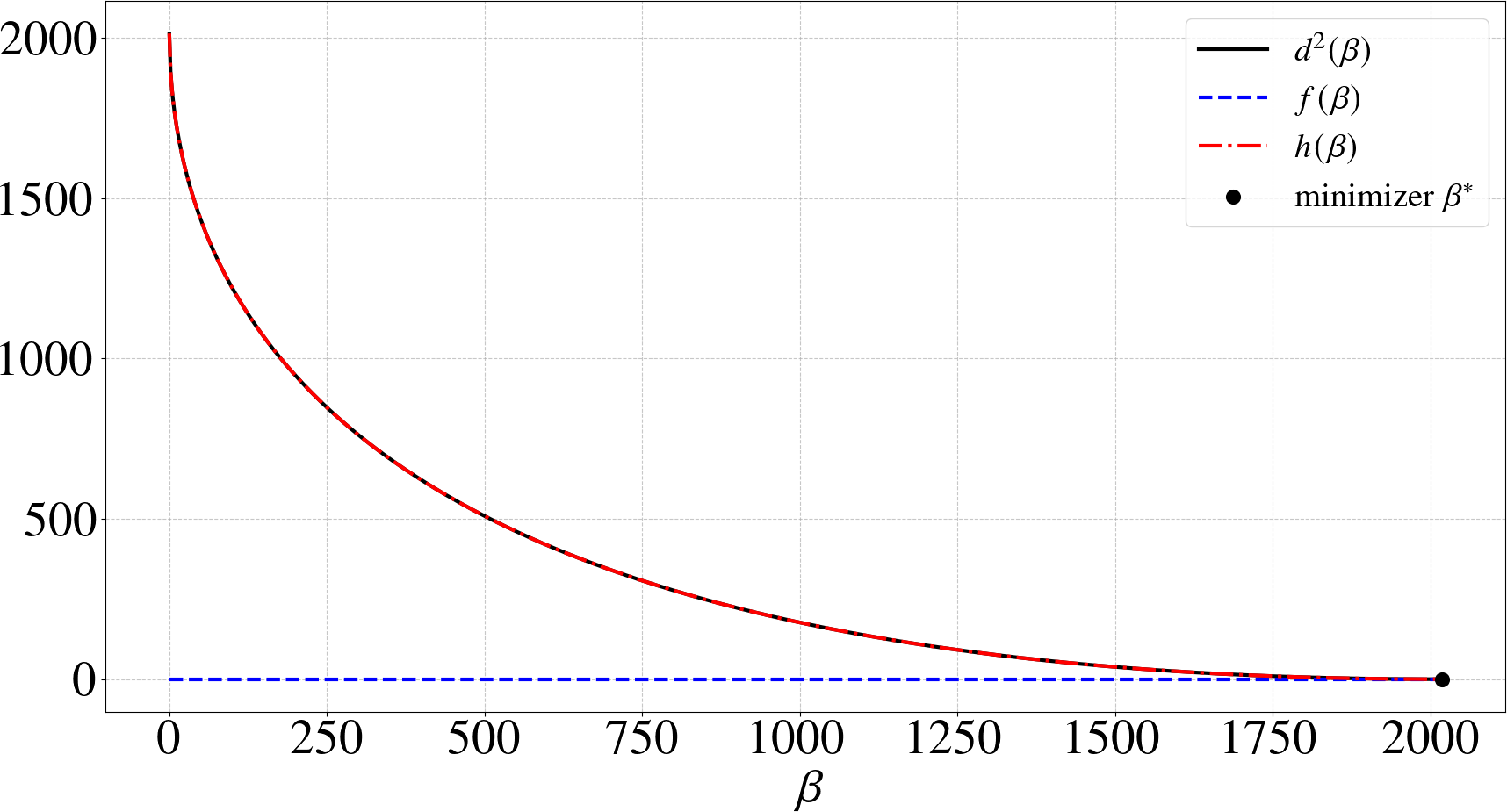}\\
\end{tabularx}
\end{center}
\caption{The functions $d^2(\beta)$, $f(\beta)$, and $h(\beta)$ over the search interval $I$, with the minimizer $\beta^\ast$ of $d^2(\beta)$ marked. Left: manufactured out-of-bound point where the minimizer lies well inside $I$. Right: representative out-of-bound cell average from the astrophysical jet simulation ($B_0 = \sqrt{2000}$, $\varepsilon = 10^{-6}$); the minimizer lies near the upper boundary of $I$, consistent with a slight pressure violation under a strong magnetic field.}
\label{fig:slicing_validation}
\end{figure}

\subsection{Smooth Alfv\'en wave}
The circularly polarized Alfv\'en wave is an exact non-polynomial solution of the MHD equations, widely used as a benchmark for verifying the accuracy of MHD solvers~\cite{toth2000b,rossmanith2006unstaggered}. In this smooth test, the base $\IP^2$ DG scheme achieves the expected optimal convergence rate, and the optimization-based cell average limiter is not triggered. 
\par
Let the computational domain be $[0, \sqrt{5}/2]\times [0, \sqrt{5}]$ with periodic boundary conditions on all sides, and $\gamma = 5/3$.
The wave propagates at angle $\theta = \tan^{-1}(0.5)$ to the $x$-axis, and $\zeta = x\cos\theta + y\sin\theta$ denotes the coordinate along the propagation direction. The initial conditions are $\rho^0 = 1$, $p^0 = 0.1$, and
\begin{align*}
v_\parallel^0 &= 0, & v_\perp^0 &= 0.1\sin(2\pi\zeta), & v_z^0 &= 0.1\cos(2\pi\zeta), \\
B_\parallel^0 &= 1, & B_\perp^0 &= 0.1\sin(2\pi\zeta), & B_z^0 &= 0.1\cos(2\pi\zeta).
\end{align*}
For testing convergence rates, the time step is fixed at $\Delta t = (0.08/\sqrt{5})\,\Delta x$, independent of the instantaneous wave speed.
\par
The Alfv\'en speed is $B_\parallel^0/\sqrt{\rho^0} = 1$. Thus, the transverse perturbations $v_\perp$, $v_z$, $B_\perp$, and $B_z$ propagate at unit speed in the $\zeta$ direction without distortion. The wave has both spatial and temporal period $1$. We choose final time $T = 2$, at which the solution completes two full periods and returns exactly to the initial state.
The exact solution at time $t$ is $\rho = 1$, $p = 0.1$, and
\begin{align*}
v_\parallel &= 0, & v_\perp &= 0.1\sin(2\pi(\zeta - t)), & v_z &= 0.1\cos(2\pi(\zeta - t)), \\
B_\parallel &= 1, & B_\perp &= 0.1\sin(2\pi(\zeta - t)), & B_z &= 0.1\cos(2\pi(\zeta - t)).
\end{align*}
Here, $v_\parallel$ and $B_\parallel$ are the components along the propagation direction $\transpose{(\cos\theta, \sin\theta, 0)}$, and $v_\perp$ and $B_\perp$ are the in-plane transverse components.
\par
Following~\cite{toth2000b}, the numerical error is measured on the four wave-carrying components $v_\perp$, $v_z$, $B_\perp$, and $B_z$. The remaining variables $\rho$, $p$, $v_\parallel$, and $B_\parallel$ are constant in time and carry no wave dynamics.
The discrete $L_h^1$ and $L_h^\infty$ errors for numerical solutions $v_{\perp,h}$, $v_{z,h}$, $B_{\perp,h}$, and $B_{z,h}$ at time $t^n$ on cell $K$ are computed following the formula in \cite{hesthaven2008nodal}. Define errors on a mesh $\setE_h$ with resolution $\Delta x$ as follows:
\begin{align*}
\mathtt{err}_{\Delta x}^1 = \sum_{K\in\setE_h} \frac{1}{4}\Big(&\norm{v_{\perp,h}-v_{\perp}}{L_h^1(K)} + \norm{v_{z,h}-v_{z}}{L_h^1(K)} \\ &+ \norm{B_{\perp,h}-B_{\perp}}{L_h^1(K)} + \norm{B_{z,h}-B_{z}}{L_h^1(K)}\Big),\\
\mathtt{err}_{\Delta x}^\infty = \max_{K\in\setE_h}\frac{1}{4}\Big(&\norm{v_{\perp,h}-v_{\perp}}{L_h^\infty(K)} + \norm{v_{z,h}-v_{z}}{L_h^\infty(K)} \\ &+ \norm{B_{\perp,h}-B_{\perp}}{L_h^\infty(K)} + \norm{B_{z,h}-B_{z}}{L_h^\infty(K)}\Big).
\end{align*}
Then, the convergence rate is evaluated by $\ln(\mathtt{err}_{\Delta x}/\mathtt{err}_{\Delta x/2})/\ln{2}$.
Table~\ref{tab:alfven_convergence} shows the $\mathtt{err}_{\Delta x}^1$ and $\mathtt{err}_{\Delta x}^\infty$ errors at final time $T = 2$. We obtain the optimal order of convergence.
\begin{table}[ht!]
\centering
\begin{tabularx}{0.9\linewidth}{@{~~}c@{~~}|c@{~~}|C@{~~}|c@{~~}|C@{~~}|c@{~~}}
\toprule
$\Delta x$ & $\Delta t$ & $\mathtt{err}_{\Delta x}^1$ & rate & $\mathtt{err}_{\Delta x}^\infty$ & rate \\
\midrule                   
$\sqrt{5}\cdot2^{-5}$ & $2^{-2}\cdot10^{-2}$ & $1.345\cdot10^{-4}$ & ---   & $2.661\cdot10^{-4}$ & ---   \\
$\sqrt{5}\cdot2^{-6}$ & $2^{-3}\cdot10^{-2}$ & $1.833\cdot10^{-5}$ & 2.876 & $3.520\cdot10^{-5}$ & 2.918 \\
$\sqrt{5}\cdot2^{-7}$ & $2^{-4}\cdot10^{-2}$ & $2.421\cdot10^{-6}$ & 2.921 & $4.516\cdot10^{-6}$ & 2.963 \\
$\sqrt{5}\cdot2^{-8}$ & $2^{-5}\cdot10^{-2}$ & $3.125\cdot10^{-7}$ & 2.954 & $5.619\cdot10^{-7}$ & 3.007 \\
\bottomrule
\end{tabularx}
\caption{Smooth Alfv\'en wave. The errors and convergence rates for the $\IP^2$ DG scheme.}
\label{tab:alfven_convergence}
\end{table}

\subsection{Rotor problem}
The rotor problem~\cite{balsara1999staggered,ciucua2020implicit} describes a rapidly rotating dense disk of fluid embedded in a lighter ambient medium permeated by a uniform magnetic field.
As the disk evolves, the magnetic field brakes its rotation and launches torsional Alfv\'{e}n waves into the surrounding medium, providing a robust test due to the sharp density contrast and the development of low-pressure regions.
\par
The computational domain is $\Omega = [0,1]^2$ with outflow boundary conditions on $\partial\Omega$.
The initial conditions are $(p, u_3, B_1, B_2, B_3) = (0.5, 0, 2.5/\sqrt{4\pi}, 0, 0)$ and
\begin{align*}
(\rho, u_1, u_2) =
\begin{cases}
(10, -(y-0.5)/r_0, (x-0.5)/r_0)                   & \text{if}~ r \leq r_0, \\
(1+9\lambda, -\lambda(y-0.5)/r, \lambda(x-0.5)/r) & \text{if}~ r_0 < r \leq r_1, \\
(1, 0, 0)                                         & \text{if}~ r > r_1,
\end{cases}
\end{align*}
where $r = \sqrt{(x-0.5)^2+(y-0.5)^2}$, $r_0 = 0.1$, $r_1 = 0.115$, and $\lambda = (r_1-r)/(r_1-r_0)$. The adiabatic index $\gamma = 5/3$.
The domain is uniformly partitioned into a $300\times300$ mesh. The numerical admissible set tolerance is $\varepsilon = 10^{-9}$.
Figure~\ref{fig:Rotor} shows the density $\rho$, the thermal pressure $p$, and the Mach number $\abs{\vec{u}}/c_s$ with $c_s = \sqrt{\gamma p/\rho}$, at the final time $T = 0.295$.
The density snapshot shows the initially circular disk compressed into an elongated ring-shaped structure along the direction of the initial field $B_1$. The Mach number snapshot displays the torsional Alfv\'{e}n waves launched into the ambient medium.
The results agree with those in~\cite{wu2022provably}. This test validates the base scheme and the cell average limiter is not triggered.
\begin{figure}[ht!]
\begin{center}
\begin{tabularx}{\linewidth}{@{}c@{~}c@{~}c@{~}c@{~}c@{~}c@{}}
\includegraphics[width=0.24\textwidth]{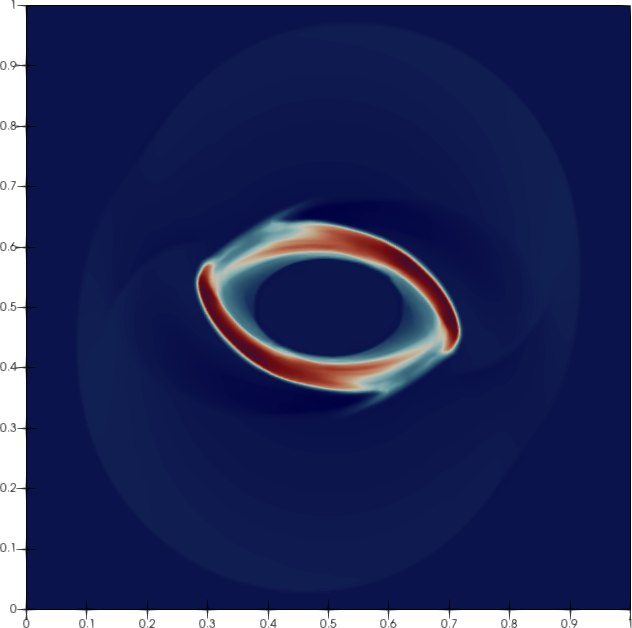} &
\includegraphics[width=0.073\textwidth]{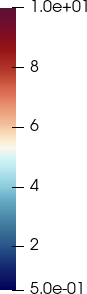} &
\includegraphics[width=0.24\textwidth]{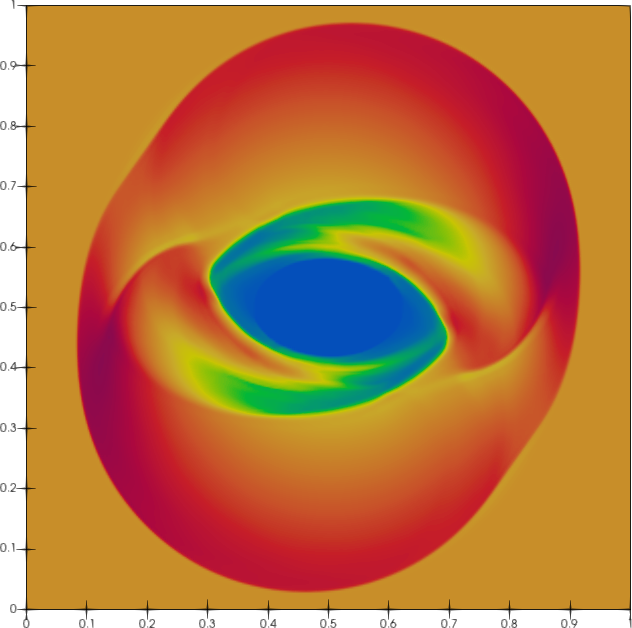} &
\includegraphics[width=0.0715\textwidth]{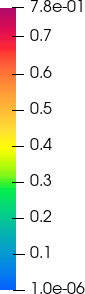} &
\includegraphics[width=0.24\textwidth]{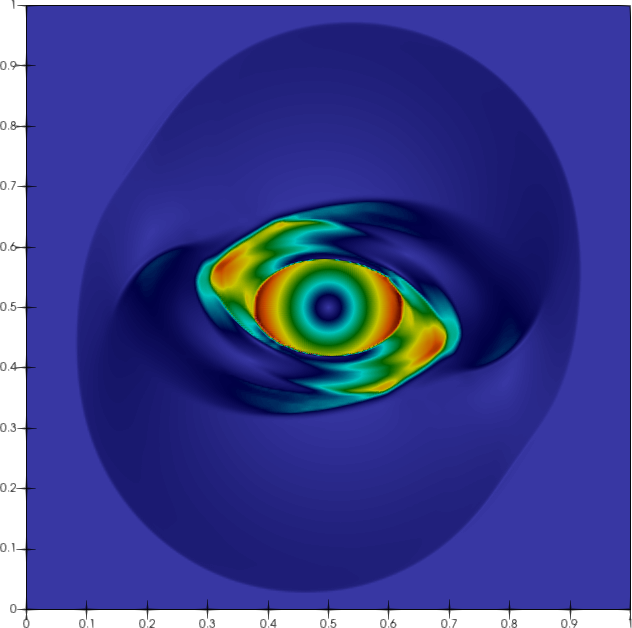} &
\includegraphics[width=0.076\textwidth]{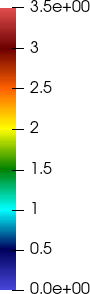} \\
\end{tabularx}
\end{center}
\caption{Rotor problem. Snapshots at $T = 0.295$. From left to right: density, thermal pressure, and Mach number.}
\label{fig:Rotor}
\end{figure}

\subsection{Orszag--Tang test}\label{sec:numerical_experiments:OT}
The Orszag--Tang vortex problem~\cite{toth2000b} is a classical benchmark for MHD solvers that has been widely used to assess numerical methods~\cite{wu2019provably,wu2025high}.
Although the initial conditions are smooth, multiple shocks develop and interact as time evolves, producing a complex turbulent-like flow structure that challenges the robustness of the numerical scheme.
The simulation runs to the final time $T = 3$, providing a stringent test of our algorithm to preserve admissibility over long-time integration.
\par
We choose the computational domain $\Omega = [0,2\pi]^2$ with periodic boundary conditions on all sides. The initial conditions are
\begin{align*}
\rho^0 = \gamma^2, &&
\vec{u}^0 = \begin{bmatrix}
-\sin{y} \\
\sin{x} \\
0
\end{bmatrix}, &&
\vec{B}^0 = \begin{bmatrix}
-\sin{y} \\
\sin{2x} \\
0
\end{bmatrix}, &&
\text{and}\quad
p^0 = \gamma,
\end{align*}
where the adiabatic index $\gamma = 5/3$. The domain $\Omega$ is uniformly partitioned into a $450 \times 450$ mesh with resolution $\Delta{x} = \pi/225$. The CFL coefficient is $0.7$ and the numerical admissible set tolerance is $\varepsilon = 10^{-9}$.
\par
Figure~\ref{fig:OT_solution} shows the snapshots of density and magnitude of magnetic field at time $t = 1, 2, 3$.
At $t = 1$, the vortex structure remains relatively coherent and the first shocks are beginning to form. By $t = 2$, multiple shocks have developed and begun to interact, and fine-scale features are visible in both the density and the magnitude of magnetic field. At $t = 3$, a fully complex multi-shock pattern has emerged in both fields, consistent with the expected transition toward turbulence. 
The results are in good agreement with those reported in the literature~\cite{wu2019provably,wu2025high,toth2000b}.
\par
The CFL coefficient $0.7$ is deliberately chosen well beyond the linear stability CFL to stress the scheme and trigger the cell average limiter.
Figure~\ref{fig:OT_DY} shows the total number of DY iterations per time step (left) and the convergence of the DY method in the first RK stage at a representative time (right). To measure convergence, we run the DY method for sufficiently many iterations to approximate the minimizer $\vecc{X}^\ast$ and the fixed point $\vecc{Z}^\ast$ numerically. 
Each minimization problem involves $1.215$ million unknowns ($6 \times 450 \times 450$), yet converges in only a few steps with error monotonically decreasing below the tolerance $\varepsilon = 10^{-9}$. Asymptotic linear convergence is observed, consistent with the optimization-based limiter applied to the Euler equations and compressible Navier--Stokes equations~\cite{liu2025efficient,liu2024optimization}.
\begin{figure}[ht!]
\begin{center}
\begin{tabularx}{0.9\linewidth}{@{}c@{~}c@{~}c@{~}c@{~}c@{}}
\begin{sideways}{\hspace{0.8cm} density $\rho$}\end{sideways} &
\includegraphics[width=0.26\textwidth]{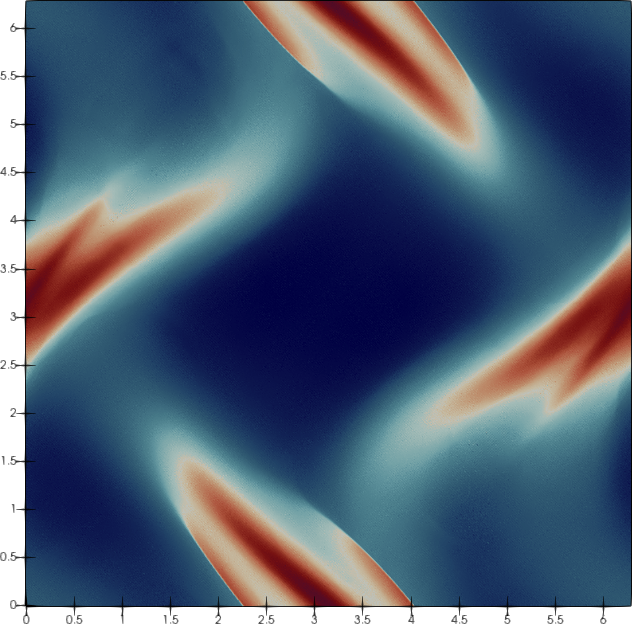} &
\includegraphics[width=0.26\textwidth]{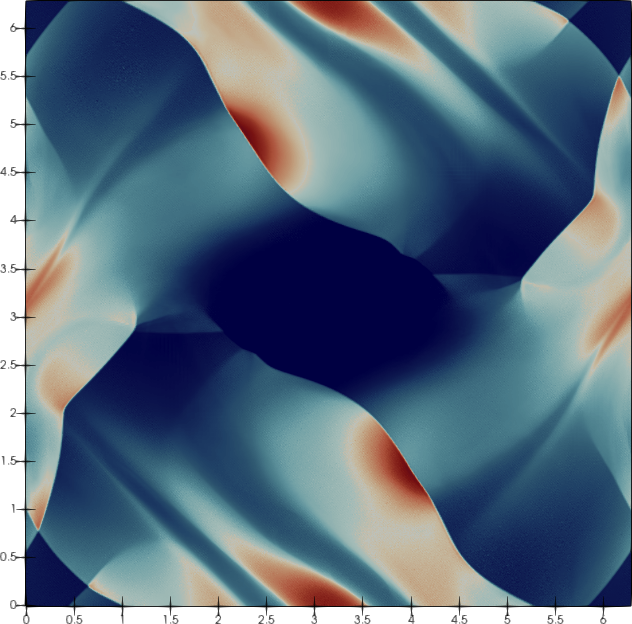} &
\includegraphics[width=0.26\textwidth]{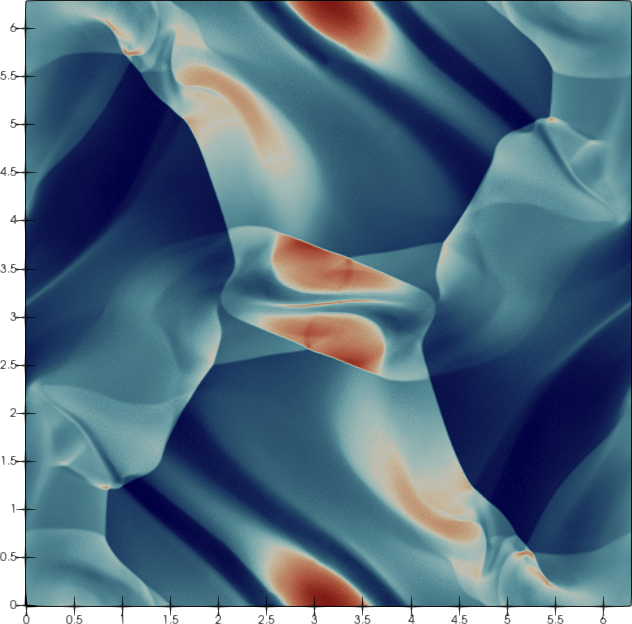} &
\includegraphics[width=0.08\textwidth]{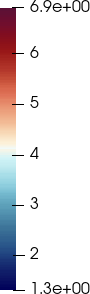} \\
\begin{sideways}{\hspace{0.4cm} magnitude of $\vec{B}$}\end{sideways} &
\includegraphics[width=0.26\textwidth]{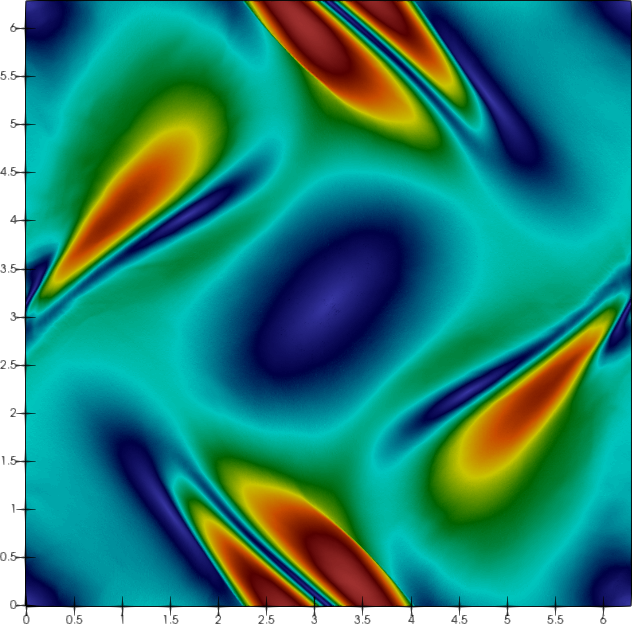} &
\includegraphics[width=0.26\textwidth]{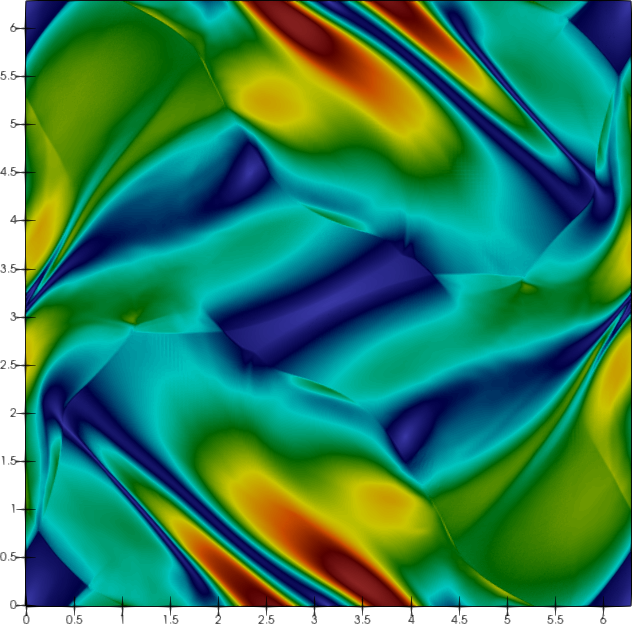} &
\includegraphics[width=0.26\textwidth]{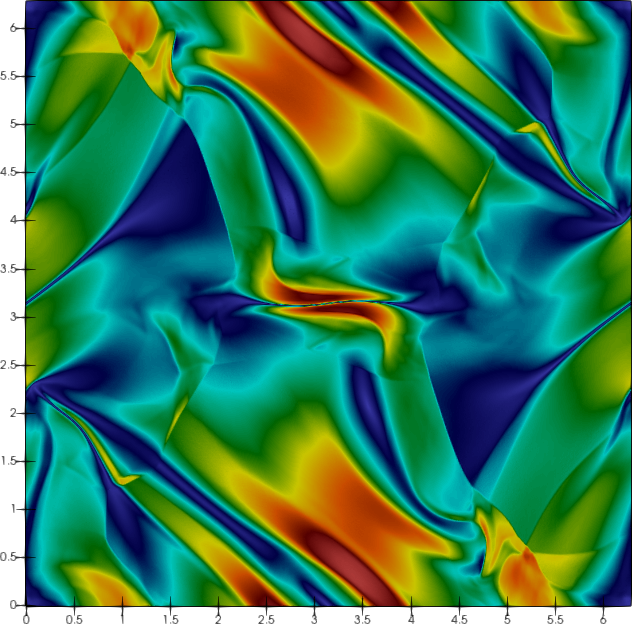} &
\includegraphics[width=0.08\textwidth]{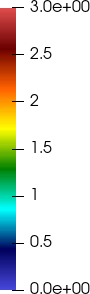} \\
~ & $t = 1$ & $t = 2$ & $t = 3$ & ~
\end{tabularx}
\end{center}
\caption{Orszag--Tang test. The snapshots of density field and magnitude of magnetic field are taken at $t = 1,2,3$. The cell average limiter is triggered. The complex wave structures, including multiple interacting shocks, are well resolved.}
\label{fig:OT_solution}
\end{figure}
\begin{figure}[ht!]
\begin{center}
\begin{tabularx}{0.75\linewidth}{@{}c@{\quad}c@{}}
\includegraphics[width=0.35\textwidth]{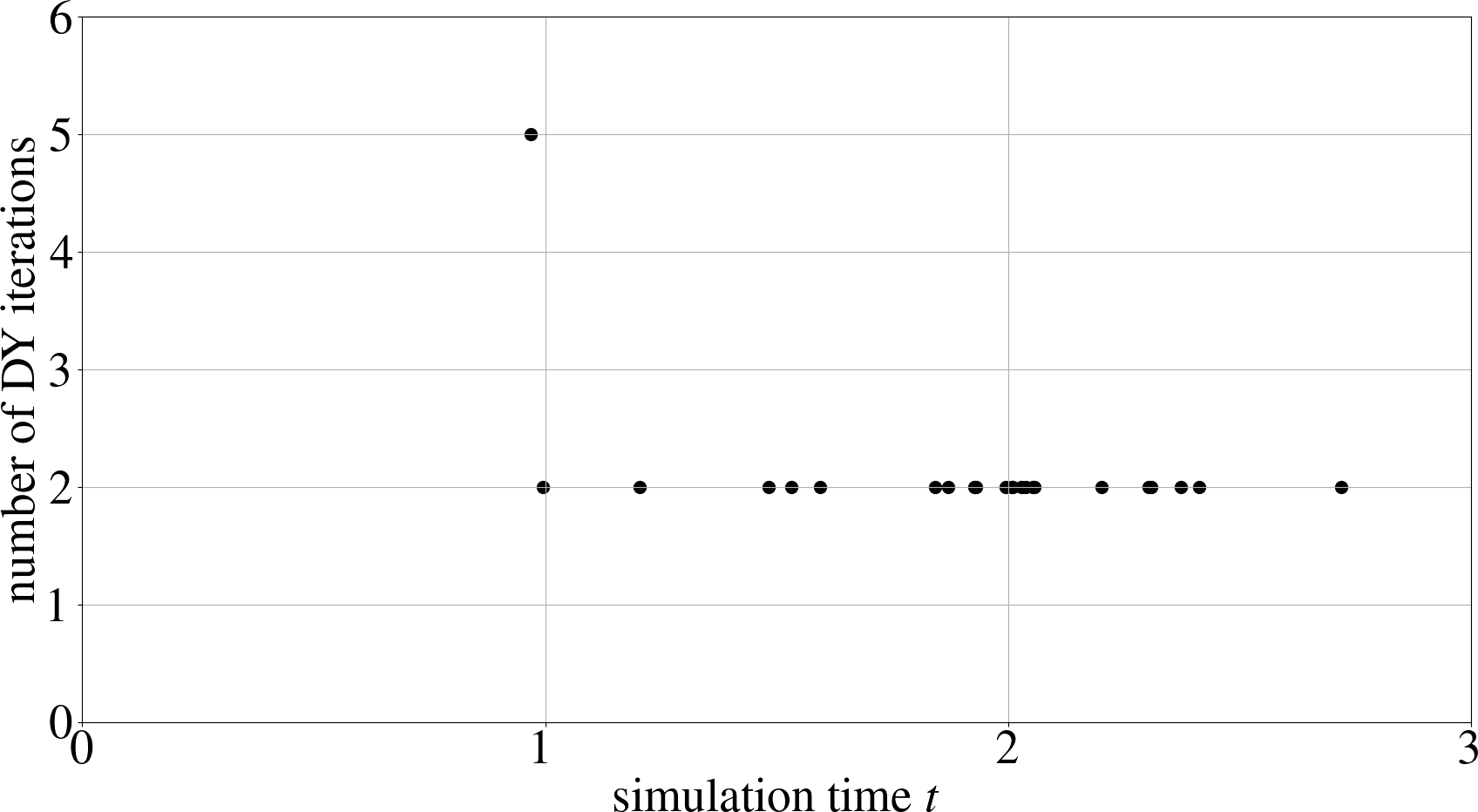} & 
\includegraphics[width=0.35\textwidth]{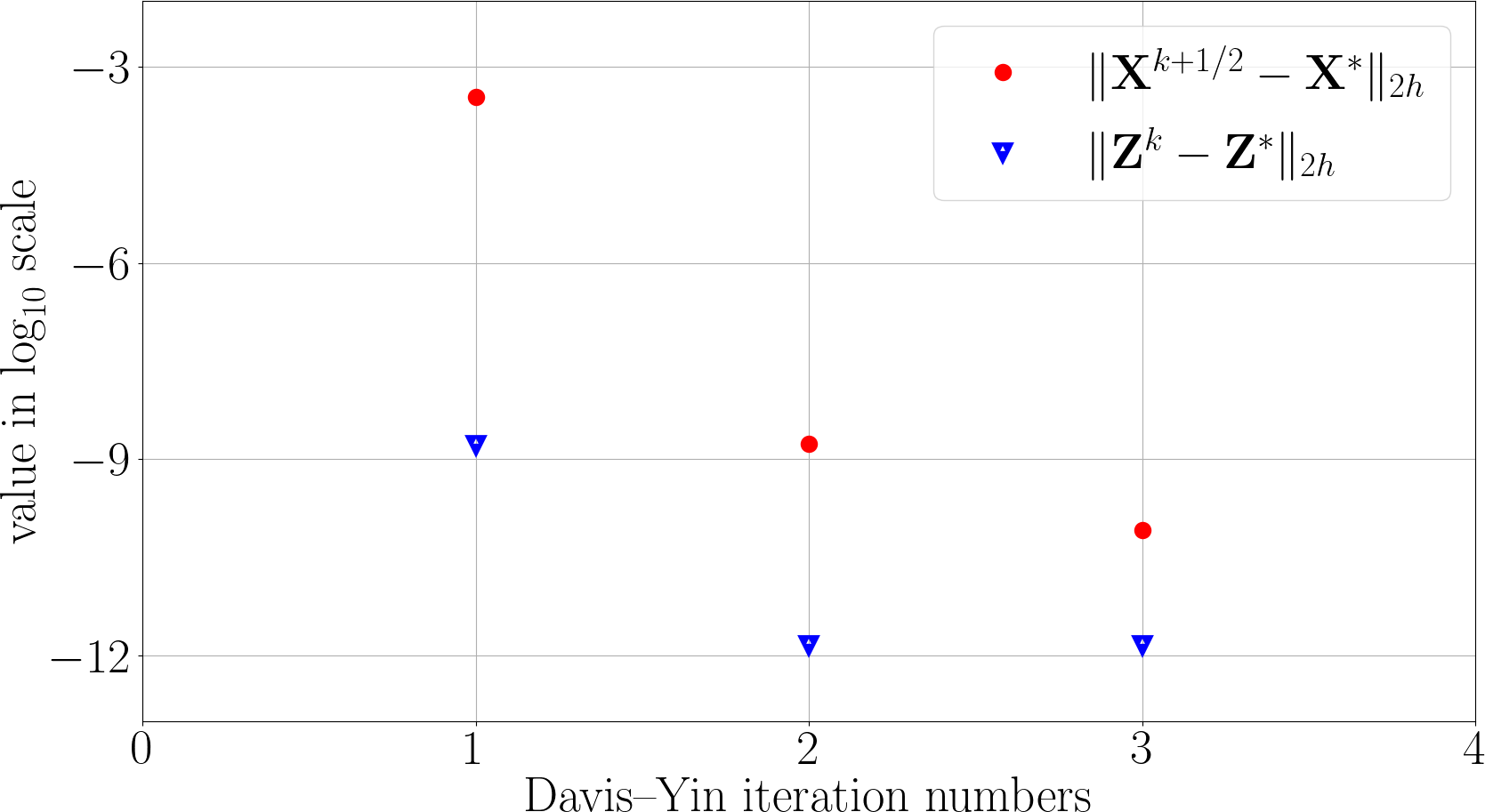} \\
\end{tabularx}
\end{center}
\caption{Left: total DY iterations per time step summed over all three RK stages. The limiter is triggered at $23$ time steps throughout the simulation. Right: convergence of the DY iteration in the first RK stage at time $t \approx 0.969$.}
\label{fig:OT_DY}
\end{figure}

\subsection{High Mach number astrophysical jet}\label{sec:numerical_experiments:astro_jet}
We simulate a Mach $800$ astrophysical jet governed by the ideal MHD equations with $\gamma = 1.4$. The jet propagates at extremely high speed into a low-density, low-pressure ambient medium, producing strong shocks and complex magnetized flow structures. Both density and pressure can easily become negative, making this a challenging benchmark for invariant-domain-preserving solvers~\cite{wu2022provably,wu2019provably}. The simulation crashes easily without the optimization-based cell average limiter.
\par
The computational domain is $\Omega = [0, 1.5]\times [0, 0.75]$ with final time $T = 0.002$. Using the reflective symmetry of the problem about the $x$-axis, results are reflected to the domain $[0, 1.5] \times [-0.75, 0.75]$.
The domain is initially filled with the ambient plasma $(\rho^0, \vec{u}^0, p^0) = (0.14, \vec{0}, 1)$. The magnetic field is initialized as $\vec{B}^0 = \transpose{(B_0, 0, 0)}$ over the entire domain.
We consider two cases: $B_0 = \sqrt{200}$ with plasma beta $2p^0/\norm{\vec{B}^0}{2}^2 = 10^{-2}$, and $B_0 = \sqrt{2000}$ with plasma beta $10^{-3}$.
On the left boundary, the jet inflow condition
\begin{align*}
(\rho, u_x, u_y, u_z, B_x, B_y, B_z, p) = (1.4, 800, 0, 0, B_0, 0, 0, 1)
\end{align*}
is imposed for $\{\abs{y} \leq 0.05\}$. Let $\normal$ denote the outward normal. Outside the nozzle $\{\abs{y} > 0.05\}$, the following ambient state 
\begin{align*}
(\rho, u_x, u_y, u_z, B_x, B_y, B_z, p) = (0.14, 0, 0, 0, B_0, 0, 0, 1)
\end{align*}
is prescribed when $\vec{u}_h\cdot\normal \leq 0$ 
and the outflow condition is applied when $\vec{u}_h\cdot\normal > 0$. 
The bottom $\{y = 0\}$ uses a reflective boundary condition. The right and top boundaries are treated as outflow.
We use $\IP^2$ DG discretization. The domain $\Omega$ is uniformly partitioned into $600 \times 300$ square cells of resolution $\Delta x = 1/400$. The TVB limiter parameter is $110$ and the numerical admissible set tolerance $\varepsilon = 10^{-6}$.
\par
The optimization-based cell average limiter is triggered during the simulation whenever out-of-bound cell averages appear.
To measure convergence, we run the DY algorithm for sufficiently many iterations to approximate the minimizer $\vec{X}^\ast$ and the fixed point $\vec{Z}^\ast$ numerically. The $\norm{\vec{X}^{k+1/2} - \vec{X}^\ast}{2h}$ and $\norm{\vec{Z}^k - \vec{Z}^\ast}{2h}$ at a representative time step where the algorithm requires the most iterations are shown in the right subfigure of Figure~\ref{fig:astro_jet_DY}.
Similar to what is observed in~\cite{liu2025efficient,liu2024optimization}, the DY iteration exhibits asymptotic linear convergence, resolving the minimization problem efficiently in a small number of iterations.
The left and middle subfigures of Figure~\ref{fig:astro_jet_DY} show the total number of DY iterations per time step, summed over all three stages of the SSP RK time integrator, for both cases. The limiter is triggered at $668093$ time steps for $B_0 = \sqrt{200}$ and at $1217125$ time steps for $B_0 = \sqrt{2000}$, with at most $6$ and $8$ iterations per step, respectively.
\par
Figure~\ref{fig:astro_jet} shows the density $\rho$ and pressure $p$ in $\log_{10}$ scale, and the magnetic pressure $\frac{1}{2}\norm{\vec{B}}{2}^2$, at $T = 0.002$ for both cases. Our simulations complete without negative density or pressure throughout.
The bow shock and jet head location are consistent with those in~\cite{wu2022provably}, confirming the global conservation is preserved.
The optimization-based cell average limiter enforces admissibility by modifying the cell averages by the least amount necessary in $L^2$ norm and the scheme remains robust with linear stability $\mathrm{CFL}$.
\begin{figure}[ht!]
\begin{center}
\begin{tabularx}{\linewidth}{@{}c@{~}c@{~}c@{}}
\includegraphics[width=0.33\textwidth]{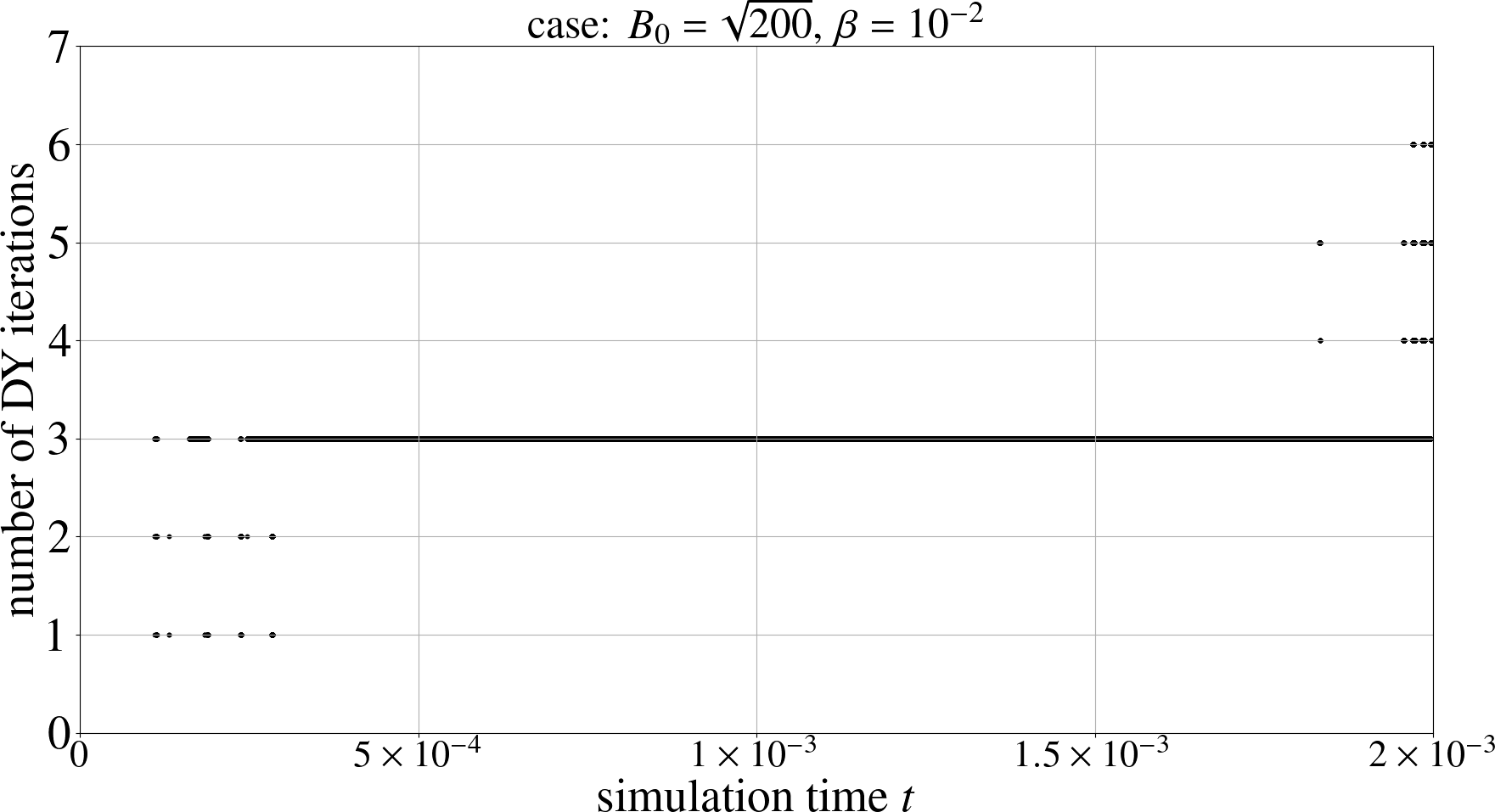} &
\includegraphics[width=0.33\textwidth]{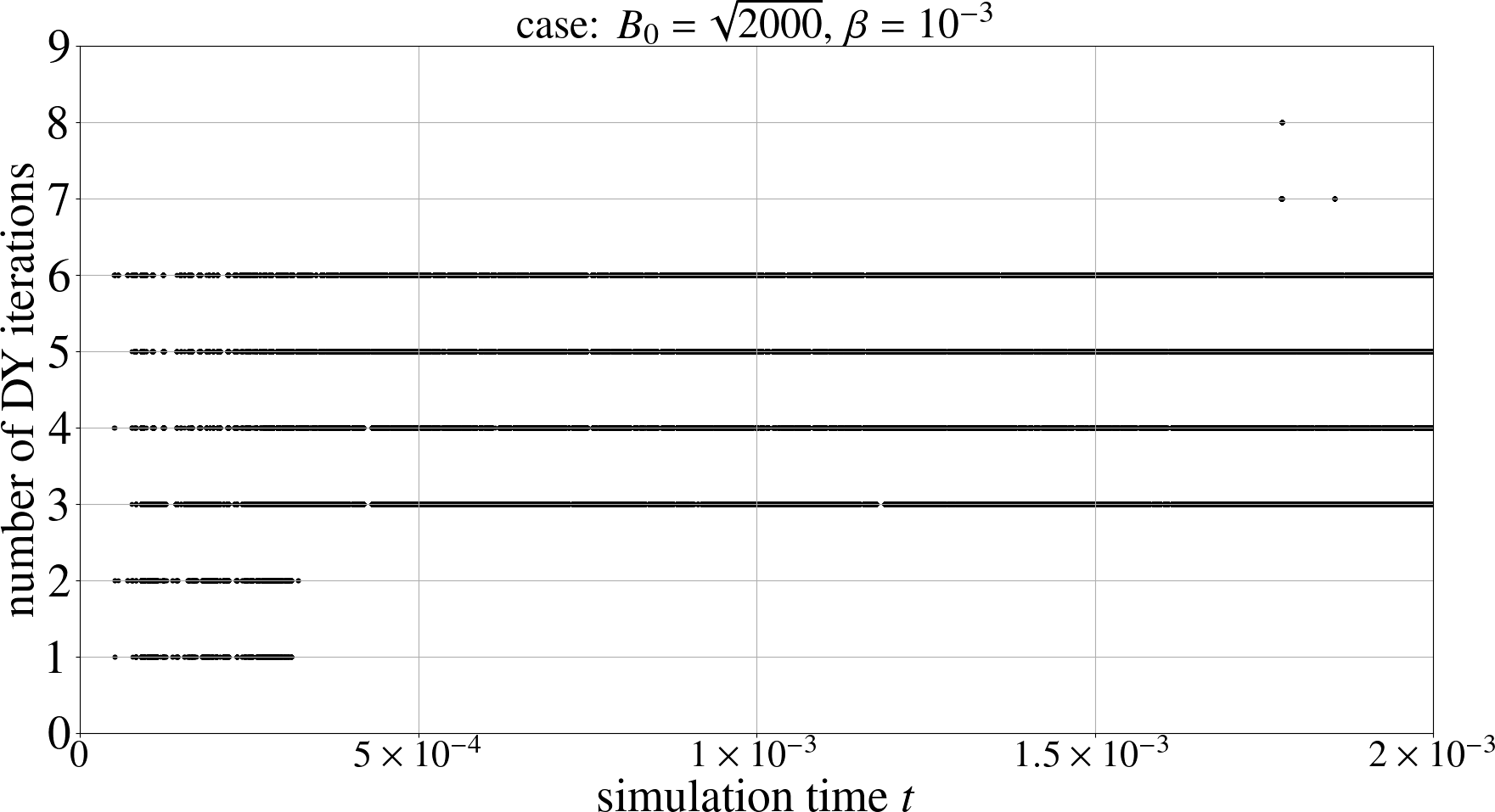} & 
\includegraphics[width=0.31\textwidth]{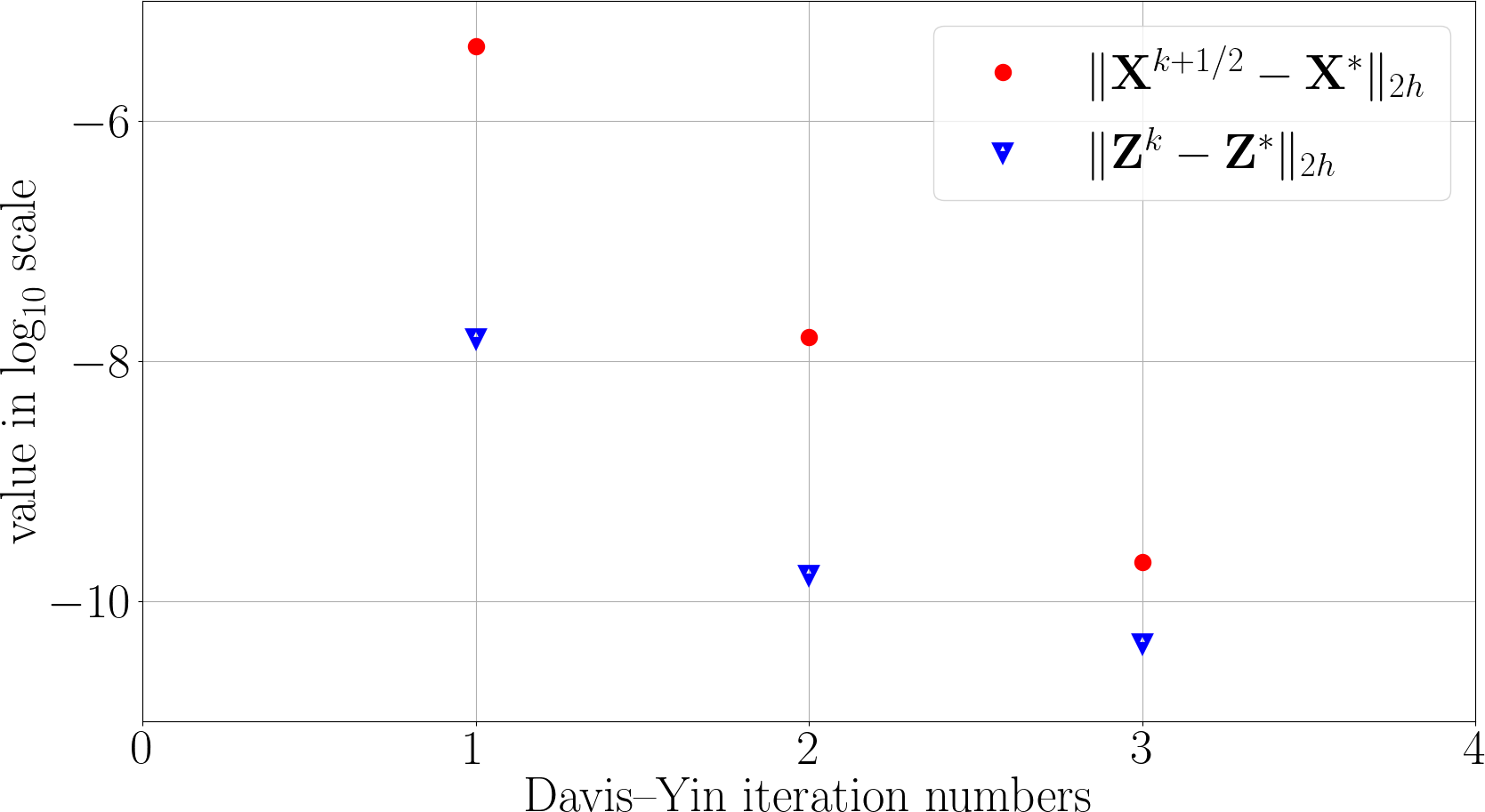} \\
\end{tabularx}
\end{center}
\caption{Left: total DY iterations per time step for $B_0 = \sqrt{200}$. Middle: total DY iterations per time step for $B_0 = \sqrt{2000}$. Right: DY convergence for case $B_0 = \sqrt{2000}$ when processing out-of-bound cell averages in the first RK stage at $t \approx 1.7760 \times 10^{-3}$. The asymptotic linear convergence is observed.}
\label{fig:astro_jet_DY}
\end{figure}
\begin{figure}[ht!]
\begin{center}
\begin{tabularx}{0.95\linewidth}{@{}c@{~}c@{\quad}c@{~}c@{}}
\includegraphics[width=0.375\textwidth]{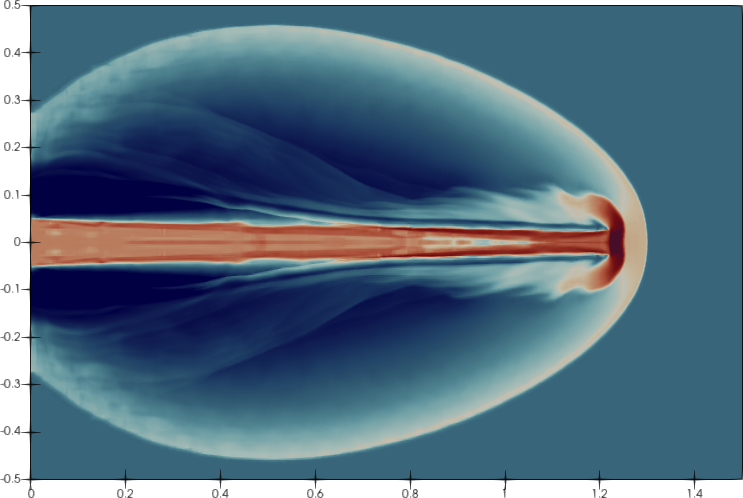} &
\includegraphics[width=0.083\textwidth]{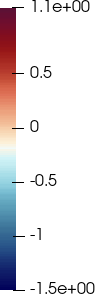} &
\includegraphics[width=0.375\textwidth]{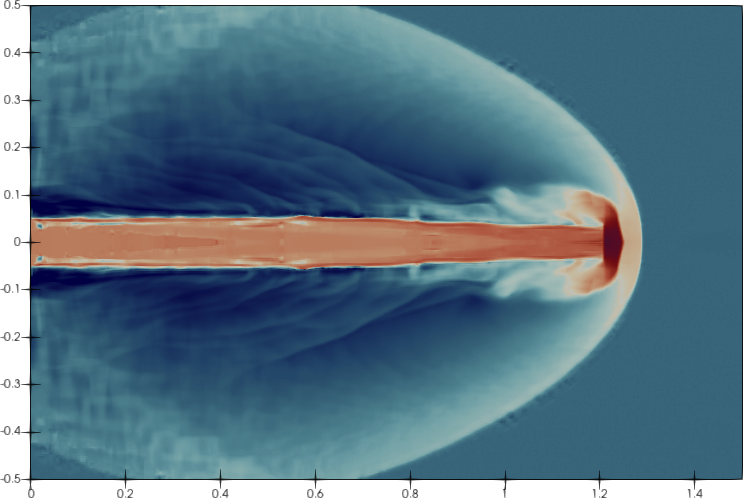} &
\includegraphics[width=0.083\textwidth]{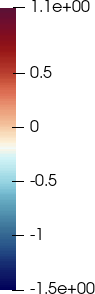} \\
\includegraphics[width=0.375\textwidth]{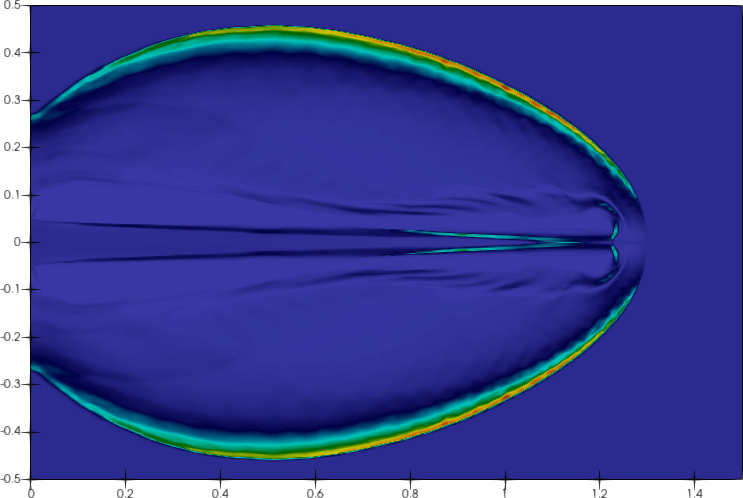} &
\includegraphics[width=0.079\textwidth]{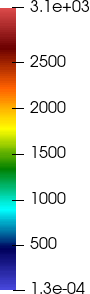} &
\includegraphics[width=0.375\textwidth]{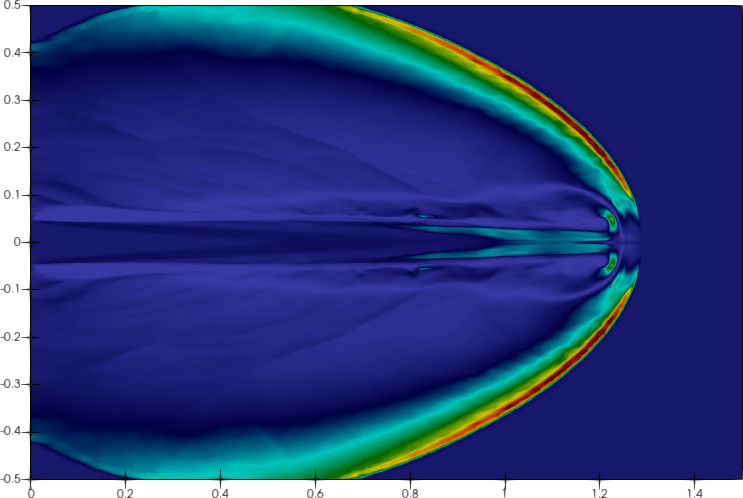} &
\includegraphics[width=0.079\textwidth]{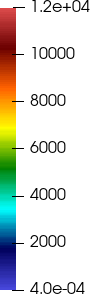} \\
\includegraphics[width=0.375\textwidth]{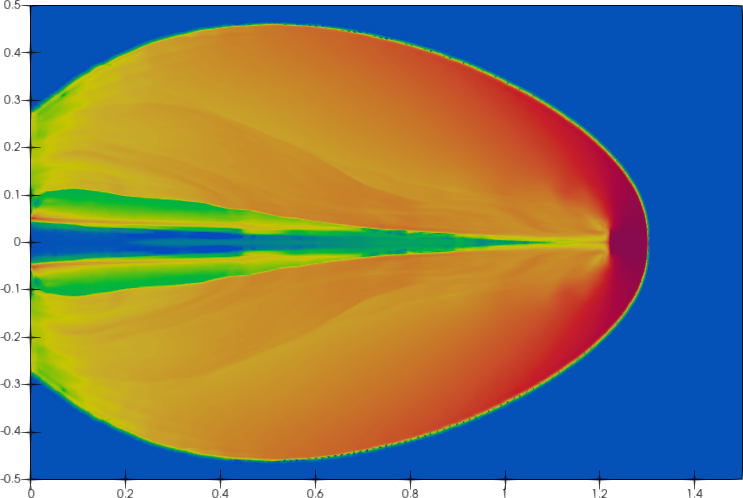} &
\includegraphics[width=0.079\textwidth]{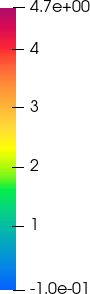} &
\includegraphics[width=0.375\textwidth]{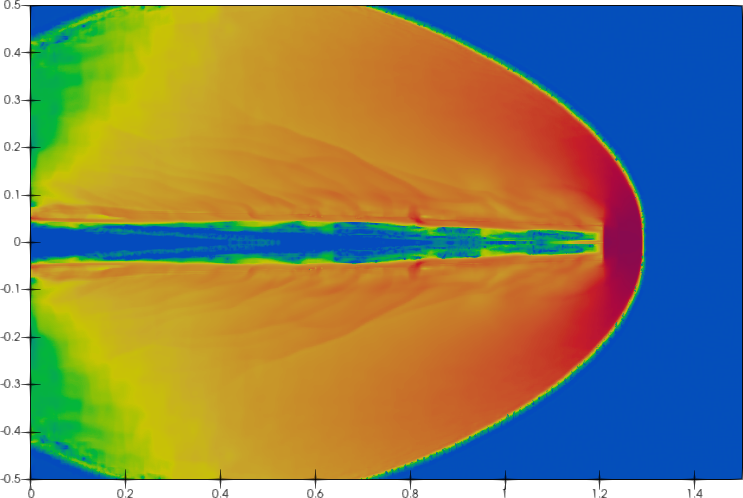} &
\includegraphics[width=0.079\textwidth]{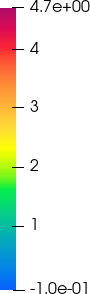} \\
\end{tabularx}
\end{center}
\caption{High Mach number astrophysical jet. Snapshots are taken at $T = 0.002$. Top to bottom: density $\rho$ in $\log_{10}$ scale, magnetic pressure $\frac{1}{2}\norm{\vec{B}}{2}^2$, and pressure $p$ in $\log_{10}$ scale. Left: $B_0 = \sqrt{200}$ with plasma beta $10^{-2}$. Right: $B_0 = \sqrt{2000}$ with plasma beta $10^{-3}$. The stronger magnetic field (plasma beta $10^{-3}$) produces a more collimated jet structure.}
\label{fig:astro_jet}
\end{figure}

\subsection{Efficiency of the slicing algorithm}
To assess the efficiency of the slicing algorithm on physically realistic test data, we extract $10^6$ out-of-bound cell averages from each of the two astrophysical jet simulations in Section~\ref{sec:numerical_experiments:astro_jet}, with $B_0 = \sqrt{200}$ and $B_0 = \sqrt{2000}$, on a $600\times 300$ mesh with numerical admissible set tolerance $\varepsilon = 10^{-6}$.
The Brent method is terminated with absolute tolerance $10^{-14}$ and relative tolerance $10^{-12}$.
All projected points are confirmed admissible in both cases.
\par
Figure~\ref{fig:slicing_efficiency} shows histograms of the number of calls to the Euler admissible set projection per MHD projection.
For $B_0 = \sqrt{200}$, the call count ranges from $16$ to $35$ with mean $27.34$ and standard deviation $2.40$.
For $B_0 = \sqrt{2000}$, the call count ranges from $16$ to $32$ with mean $27.00$ and standard deviation $3.23$.
The total wall-clock time for $10^6$ projections is approximately $5.82$ microseconds per MHD projection for $B_0=\sqrt{200}$ and $5.22$ microseconds per MHD projection for $B_0=\sqrt{2000}$, measured on a single core of an AMD EPYC 7742 processor (Bridges-2 cluster, PSC).
In these tests, each MHD projection requires only a small number of evaluations of the Euler admissible set projection, and each such evaluation is computed in closed form.
\begin{figure}[ht!]
\begin{center}
\begin{tabularx}{0.975\linewidth}{@{}c@{~}c@{}}
\includegraphics[width=0.48\textwidth]{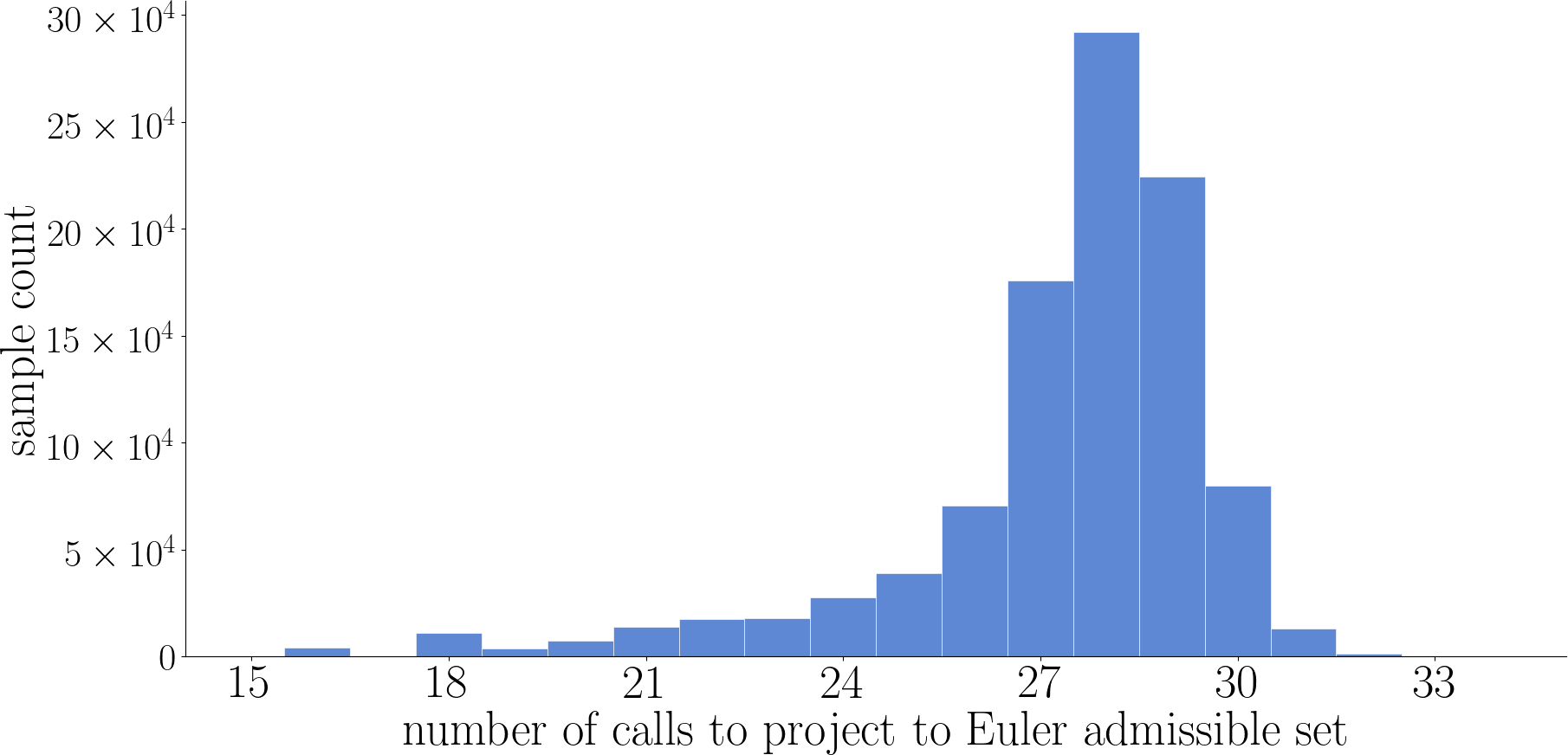} &
\includegraphics[width=0.48\textwidth]{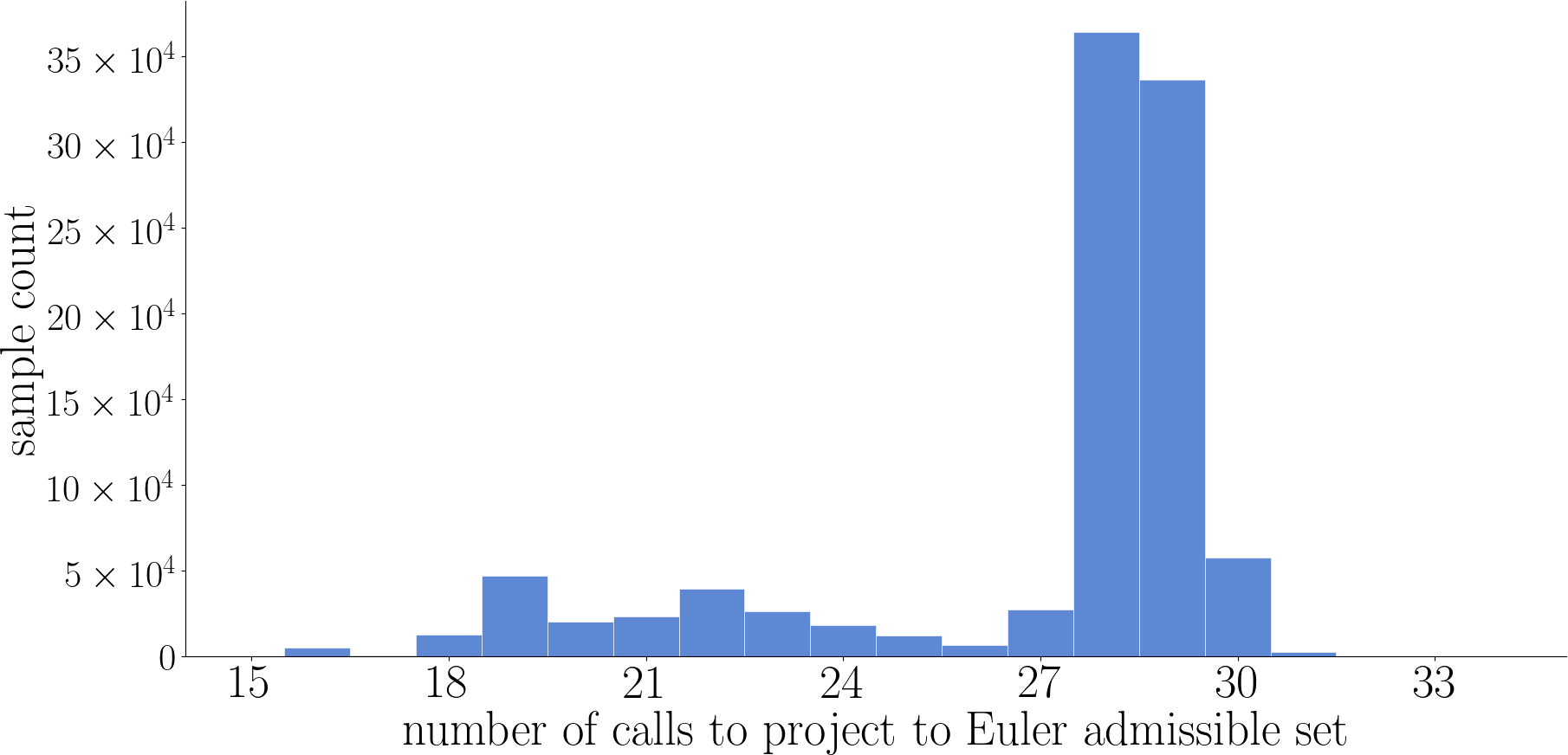} \\
\end{tabularx}
\end{center}
\caption{Histograms of the number of calls to the Euler admissible set projection per MHD projection. A total of $10^6$ cell averages outside the MHD admissible set are extracted from each astrophysical jet simulation. The $y$-axis records the number of out-of-bound cell averages that required each particular number of Euler admissible set projection calls. Left: the case of $B_0 = \sqrt{200}$. Right: the case of $B_0 = \sqrt{2000}$.}
\label{fig:slicing_efficiency}
\end{figure}

\section{Conclusions}\label{sec:conclusions}
We have developed an optimization-based cell average limiter that restores admissibility for high-order DG schemes applied to the ideal MHD equations. Our main theoretical contribution is the reduction of the projection onto the MHD admissible set $G^\varepsilon$ to a one-dimensional minimization in the magnetic energy $\beta = \norm{\vec{B}}{2}^2$, with a strictly convex and continuous reduced objective on a bounded search interval. We prove that the limiter preserves the order of accuracy, and numerical experiments show robust performance on demanding MHD benchmarks.
\par
Our limiter addresses admissibility restoration only. A remaining issue is how to combine it with a discrete divergence-free preservation scheme in a structure-preserving and conservative way. It would also be interesting to extend the framework to resistive MHD and relativistic MHD, where the admissible set has additional structure.

\section*{Acknowledgments}
The authors thank Mr. Dionysis Milesis (Boston University) for a useful discussion on estimating a lower bound of the search interval for applying the Brent method.
This work used Bridges-2 at Pittsburgh Supercomputing Center through allocation MTH260026 from the Advanced Cyberinfrastructure Coordination Ecosystem: Services \& Support (ACCESS) program, which is supported by National Science Foundation (NSF) grants \pound2138259, \pound2138286, \pound2138307, \pound2137603, and \pound2138296.

\appendix

\section{Projection onto set \texorpdfstring{$F_\beta^\varepsilon$}{F\_beta\_eps}}\label{sec:appendix_projection_F_beta} 
The derivation of the closed-form projection formula onto the compressible Euler-like admissible set follows~\cite[Appendix]{liu2025efficient} for one to three space dimensions. Here, we present the two-dimensional case for brevity. The three-dimensional case can be derived in a similar way.
\par
Given $(u, v_1, v_2, w) \notin F_\beta^\varepsilon$, we derive the projection point using the KKT conditions.
Without loss of generality, let us assume $\abs{v_1} \geq \abs{v_2}$. The resulting algorithm is summarized below.
\begin{itemize}[itemsep=3pt,leftmargin=1.75em]
\item The point $(\varepsilon, v_1, v_2, w)$ is a candidate, if $u < \varepsilon$ and $w - \frac{1}{2\varepsilon}(v_1^2 + v_2^2) \geq \varepsilon + \frac{\beta}{2}$.
\item The point $(\varepsilon, 0, 0, \varepsilon + \frac{\beta}{2})$ is a candidate, if $u < \varepsilon$, $v_1 = 0$, $v_2 = 0$, and $w < \varepsilon + \frac{\beta}{2}$.
\item If $v_1, v_2 \neq0$, solve the cubic equation $am_1^3 + (4\varepsilon^2 + \varepsilon\beta - 2\varepsilon w)m_1 - 2\varepsilon^2 v_1 = 0$ with $a = 1 + v_2^2/v_1^2$ to obtain all real roots. Examine each real root individually. 
Let $m_{1r}$ denote a real root. Then the point $(\varepsilon, m_{1r}, \frac{v_2}{v_1}m_{1r}, \frac{a}{2\varepsilon}m_{1r}^2 + \varepsilon + \frac{\beta}{2})$
is a candidate, if $\frac{v_1}{m_{1r}} > 1$ and $2\varepsilon u + a m_{1r}(v_1 - m_{1r}) < 2\varepsilon^2$. 
\item The point $(u, 0, 0, \varepsilon + \frac{\beta}{2})$ is a candidate, if $u\geq\varepsilon$, $v_1 = 0$, $v_2 = 0$, and $w < \varepsilon + \frac{\beta}{2}$.
\item Compute $\rho_1$ and $\rho_2$ using the following formulas only if they are real.    
\begin{align*}
\rho_1 &= \frac{u}{2} + \frac{1}{2} \frac{\sqrt{u^2(2v_1^2 + \frac{1}{a}(\varepsilon+\frac{\beta}{2} + u - w)^2) - 2uv_1^2(w-\varepsilon-\frac{\beta}{2}) + av_1^4}}{\sqrt{2v_1^2 + \frac{1}{a}(\varepsilon+\frac{\beta}{2} + u - w)^2}},\\
\rho_2 &= \frac{u}{2} - \frac{1}{2} \frac{\sqrt{u^2(2v_1^2 + \frac{1}{a}(\varepsilon+\frac{\beta}{2} + u - w)^2) - 2uv_1^2(w-\varepsilon-\frac{\beta}{2}) + av_1^4}}{\sqrt{2v_1^2 + \frac{1}{a}(\varepsilon+\frac{\beta}{2} + u - w)^2}}.
\end{align*}
For real values of $\rho_1$ and $\rho_2$, compute $m_{1\alpha}$ and $m_{1\beta}$ using the following formulas only if they are real.
\begin{align*}
m_{1\alpha}(\rho) &= \frac{1}{2}v_1 - \frac{1}{2a}\sqrt{-8a\rho^2 + 8au\rho + a^2 v_1^2} \\
\text{and}\quad
m_{1\beta}(\rho) &= \frac{1}{2}v_1 + \frac{1}{2a}\sqrt{-8a\rho^2 + 8au\rho + a^2 v_1^2}.
\end{align*}
Then, for real points:
\begin{itemize}[itemsep=3pt,leftmargin=1.75em]
\item The point $(\rho_1, m_{1\alpha}(\rho_1), \frac{v_2}{v_1}m_{1\alpha}(\rho_1), \varepsilon+\frac{\beta}{2} + a\frac{m_{1\alpha}(\rho_1)^2}{2\rho_1})$ is a candidate, if $\rho_1\geq\varepsilon$ and $\varepsilon+\frac{\beta}{2} + a\frac{m_{1\alpha}(\rho_1)^2}{2\rho_1} > w$.
\item The point $(\rho_2, m_{1\alpha}(\rho_2), \frac{v_2}{v_1}m_{1\alpha}(\rho_2), \varepsilon+\frac{\beta}{2} + a\frac{m_{1\alpha}(\rho_2)^2}{2\rho_2})$ is a candidate, if $\rho_2\geq\varepsilon$ and $\varepsilon+\frac{\beta}{2} + a\frac{m_{1\alpha}(\rho_2)^2}{2\rho_2} > w$.
\item The point $(\rho_1, m_{1\beta}(\rho_1), \frac{v_2}{v_1}m_{1\beta}(\rho_1), \varepsilon+\frac{\beta}{2} + a\frac{m_{1\beta}(\rho_1)^2}{2\rho_1})$ is a candidate, if $\rho_1\geq\varepsilon$ and $\varepsilon+\frac{\beta}{2} + a\frac{m_{1\beta}(\rho_1)^2}{2\rho_1} > w$.
\item The point $(\rho_2, m_{1\beta}(\rho_2), \frac{v_2}{v_1}m_{1\beta}(\rho_2), \varepsilon+\frac{\beta}{2} + a\frac{m_{1\beta}(\rho_2)^2}{2\rho_2})$ is a candidate, if $\rho_2\geq\varepsilon$ and $\varepsilon+\frac{\beta}{2} + a\frac{m_{1\beta}(\rho_2)^2}{2\rho_2} > w$.
\end{itemize}
\end{itemize}
Pick the point from the candidates that minimizes the distance to $(u, v_1, v_2, w)$, which gives the projection point.
\begin{remark}
For numerical stability, the computation of the factor $a = 1 + v_2^2/v_1^2$ must avoid division by a small denominator. Note that the two momentum components play symmetric roles. If $\abs{v_2} > \abs{v_1}$, we instead work with the variable $m_2$. The corresponding formulas are derived analogously and omitted for brevity.
\end{remark}

\section{Brent method}\label{sec:appendix_brent} 
Consider a continuous convex function $f(x)$ on a closed interval $[a,b]$. Assume that the derivative of $f$ is not available. We utilize the Brent method \cite[Section~5]{brent2013algorithms} to compute its minimizer.
Given absolute tolerance $\epsilon_\mathrm{abs} = 10^{-14}$, relative tolerance $\epsilon_\mathrm{rel} = 10^{-12}$, and golden ratio $c = \frac{3-\sqrt{5}}{2}$, the algorithm is as follows.
\begin{algorithmic}
\STATE \textbf{Input}: lower bound $a$, upper bound $b$, and the objective function $f(x)$
\STATE \textbf{Initialize}: let $s_a = a$, $s_b=b$, and $m = \frac{1}{2}(a + b)$. Set $e = 0$, $d = 0$, and
\begin{align*}
x &= a + c(b-a), & w &= x, & v &= x, \\
f_x &= f(x), & f_w &= f_x, & f_v &= f_x.
\end{align*}\vspace*{-1.0\baselineskip}
\STATE Compute tolerance $\mathtt{tol} = \epsilon_\mathrm{rel} \abs{x} + \epsilon_\mathrm{abs}$. 
If stopping criterion $\abs{x - m} < 2\mathtt{tol} - \frac{1}{2}(s_b - s_a)$ is satisfied, then return $x$. Otherwise enter into the following iterative loop until convergence.
\LOOP
\STATE \textbf{Step 1.} Compute middle point $m = \frac{1}{2}(s_a + s_b)$. If $\abs{e} > \mathtt{tol}$, compute
\begin{align*}
r &= (x-w)(f_x - f_v), \\
q &= (x-v)(f_x - f_w), \\
p &= (x-v)q - (x-w)r,  \\
q &= 2(q-r).
\end{align*}
If $q > 0$, set $p \leftarrow -p$; else set $q \leftarrow -q$. 
Store previous step $r_e \leftarrow e$ and update $e \leftarrow d$.
\STATE \textbf{Step 2.} If $\abs{p} < \frac{1}{2}\abs{q r_e}$ and $q(s_a - x) < p < q(s_b - x)$, apply inverse parabolic interpolation. Update: $d = \frac{p}{q}$ and $u = x + d$.\\
Otherwise, apply Golden section update: if $x < m$, set $e = s_b - x$; else set $e = s_a - x$. Update $d = ce$.
\STATE \textbf{Step 3.} Enforce minimum displacement: if $\abs{d} \geq \mathtt{tol}$, set $u = x + d$; else, set $u = x + \mathtt{tol}\cdot\mathrm{sgn}(d)$.
\STATE \textbf{Step 4.} Let $f_u = f(u)$. Update searching interval and interpolation points.
\STATE \hspace{1em}\textbf{if} $f_u \leq f_x$:
\STATE \hspace{2em}Update $s_b \leftarrow x$ when $u < x$; otherwise, update $s_a\leftarrow x$ when $u \geq x$. Set
\begin{align*}
v &\leftarrow w,\quad f_v \leftarrow f_w, \\
w &\leftarrow x,\quad f_w \leftarrow f_x, \\
x &\leftarrow u,\quad f_x \leftarrow f_u.
\end{align*}\vspace*{-1.25\baselineskip}
\STATE \hspace{1em}\textbf{else}
\STATE \hspace{2em}Update $s_a\leftarrow u$ when $u < x$; otherwise, update $s_b\leftarrow u$ when $u \geq x$. Set
\begin{align*}
\text{if}~ f_u \leq f_w ~\text{or}~ w = x\!:&\quad (v,f_v) \leftarrow (w,f_w),\quad (w,f_w) \leftarrow (u,f_u), \\
\text{if}~ f_u \leq f_v ~\text{or}~ v = x ~\text{or}~ v = w\!:&\quad (v,f_v) \leftarrow (u,f_u).
\end{align*}
\vspace*{-1.25\baselineskip}
\STATE \hspace{1em}\textbf{end if}
\ENDLOOP
\end{algorithmic}

\bibliographystyle{siamplain}
\bibliography{references}
\end{document}